\documentclass[final,1p,times]{elsarticle}
\usepackage{amssymb}
\usepackage{amsthm}
\usepackage{latexsym}
\usepackage{amsmath}
\usepackage{color}
\usepackage{graphicx}
\usepackage{indentfirst}
\usepackage{mathrsfs}
\usepackage{lipsum}
\allowdisplaybreaks

\renewcommand{\thefootnote}{\arabic{footnote}}
\newtheorem{theorem}{\color{black}\indent \textbf{Theorem}}[section]
\newtheorem{lemma}{\color{black}\indent Lemma}[section]
\newtheorem{proposition}{\color{black}\indent Proposition}[section]
\newtheorem{definition}{\color{black}\indent Definition}[section]
\newtheorem{remark}{\color{black}\indent Remark}[section]

\newcommand\blfootnote[1]{%
	\begingroup
	\renewcommand\thefootnote{}\footnote{#1}%
	\addtocounter{footnote}{-1}%
	\endgroup
}			

\begin{document}

\begin{frontmatter}
\title{Central limit theorem for periodic solutions of stochastic differential equations driven by L{\'e}vy noise}

 \author{{ \blfootnote{$^{*}$Corresponding author.} Xinying Deng$^{a}$\footnote{ E-mail address : dengxy245@nenu.edu.cn},~ Yong Li$^{b,c*}$}  \footnote{E-mail address : liyong@jlu.edu.cn},
 ~ Xue Yang$^{b}$\footnote{ E-mail address : xueyang@jlu.edu.cn}.
 		\\
 {$^{a}$School of Mathematics and Statistics, Northeast Normal University,}
 	{Changchun, $130024$, P. R. China.}\\
 	{$^{b}$School of Mathematics, Jilin University,} {Changchun, $130012$, P. R. China.}\\
 {$^{c}$ Center for Mathematics and Interdisciplinary Sciences, Northeast Normal University,} {Changchun, $130024$, P. R. China. }
 }

\begin{abstract}
Through certain appropriate constructions, we establish periodic solutions in distribution for some stochastic
differential equations with infinite-dimensional L{\'e}vy noise. Additionally, we
obtain the corresponding periodic measures and periodic transition semigroup. Under suitable conditions, we
also achieve a
certain
 contractivity in the space of probability measures. By constructing an appropriate invariant measure, we standardize the
 observation functions. Utilizing the classical martingale approximation approach, we establish the law of
 large numbers and the central limit theorem.

\end{abstract}\end{frontmatter}
\section{Introduction}
L{\'e}vy processes are stochastic processes characterized by stationary and independent increments. They are particularly
valuable as they can capture discontinuous and abrupt fluctuations encountered in practical scenarios. From a
mathematical perspective, L{\'e}vy processes represent a significant class of semimartingales and Markov processes,
encompassing Wiener processes and Poisson processes as notable special cases. Therefore, L{\'e}vy processes hold
importance both in theory and practical applications. For a comprehensive treatment of the theory, and some
 computational methods, see, e.g. \cite{ref51}; for the properties of solutions to some stochastic differential
 equations, see, e.g. \cite{ref31}; for the asymptotic stability of the solutions to some semilinear stochastic
  differential equations with infinite dimensional L{\'e}vy noise, see \cite{ref40}.

As widely recognized, the strong law of large numbers and central limit theorem for ergodic Markov processes have
garnered significant attention over the decades, dating back to \cite{ref1} in 1937.
In the case of time homogeneity, the strong law of large numbers(SLLN) and central limit theorem(CLT)
demonstrate that the asymptotic behavior of an observation along a Markov
process can be characterized by the invariant measure, a concept elucidated in the earlier work by Kuksin and
Shirikyan \cite{ref22}.
But what about inhomogeneous systems?
In 1956, Dobrushin proved a definitive central limit theorem for inhomogeneous Markov chains, which played a key
 role in the research of the strong law of large numbers and central limit theorem for later generations.

 It is well known that the Markov property plays a crucial role in the study of stochastic differential equations. It emphasizes
 what will happen in the future with respect to the current state, not the past. Another important property is the time
 homogeneity, which emphasizes that the transition probability from state $i$ to state $j$ depends on the length of the time
 interval and is independent of the starting time. The strong law of large numbers and the central limit theorem for homogeneous processes were extensively studied, see, e.g., \cite{ref90}\cite{ref16}. However, there has been limited research on the inhomogeneous case. When the system lacks time homogeneity, meaning that the drift and diffusion coefficients are time-dependent, many of the advantageous properties of the system become unavailable. Plagued by these challenges, we seek to identify beneficial properties to compensate for the shortcomings arising from the absence of time homogeneity. If we can address this issue, we would be able to generalize many conclusions from time-homogeneous systems to inhomogeneous systems. Inspired by \cite{ref4}, we consider the periodicity of the system in the sense of distribution. For a periodic system driven by a L{\'e}vy process, the dynamic
  behaviors in two periods can  be entirely different due to the independent increment property of the L{\'e}vy process, in other words, they are not equal almost everywhere, but their probability distributions are the same. 

One of the aims of this article is to investigate SLLN and CLT for inhomogeneous Markov processes driven by L{\'e}vy noise.
The inconveniences of inhomogeneity can be addressed by introducing periodicity.
Since it was introduced by Khasminskii \cite{ref15}, periodic solutions in the context of periodic Markov processes
have been studied extensively. The paper \cite{ref4} explored stochastic periodic solution in distribution to
 the Fokker-Planck equation, assuming an unconventional Lyapunov condition.
With a growing focus on stochastic periodic solutions in distribution, several contributions have been made in the literature.
For instance, see, e.g., \cite{ref41} for affine periodic solutions for stochastic
 differential equations utilizing a LaSalle-type stationary oscillation principle; \cite{ref33} for periodic solutions in
 distribution for mean-field stochastic differential equations; \cite{ref5} for the ergodicity of a periodic probability
 measure for SDEs;
 \cite{ref13} for random periodic solutions for
 semilinear
 stochastic differential equations; \cite{ref24} for a unique ergodic invariant measure in the incompressible
  2D-Navier-Stokes equations with periodic boundary conditions, which is the first to develop the results in
 \cite{ref38}\cite{ref36}\cite{ref20} to the time-inhomogeneous setting on the torus with highly degenerate noise. Inspired by
  this, combining with periodic solutions in distribution, we derive a periodic measure and an invariant measure, which
 is also ergodic.
  Using the Wasserstein metric, we establish a certain compressibility in the space of the measure,
   ensuring the ergodicity of the invariant measure and thus exponential convergence under a specific class of
   observation functions. Employing the method of martingale approximation, which traced back to
   the work of Gordin and Lif{\v{s}}ic \cite{ref16}, and Kipnis and Varadhan\cite{ref23}, we achieve the desired
   SLLN and CLT by analyzing a special martingale through the classical martingale approximation method (
    the residual term is negligible, see Lemma \ref{26}).

The rest of paper is arranged as follows. In Section 2,  we introduce some notations that will be utilized consistently
throughout the paper.
In Section 3, we introduce definitions essential for the study of periodic solutions in distribution and periodic measures.
 We also give certain assumptions under which we derive moment estimates of the solutions and a unique periodic solution in the
  sense of distribution. Additionally, we establish contractivity properties on the space of probability measures, which are crucial
  for subsequent estimations.  In Section 4, we  consider a special case of the original equation. By means
 of classical martingale approximation, we derive the strong law of numbers in the space of weighted observation functions. In Section 5,
  inspired by  \cite{ref25}, we obtain the central limit theorem by validating Lindeberg-type conditions and integrating them with
   the law of large numbers for the conditioned martingale difference. For clarity, we include in the appendix
    fundamental concepts and lemmas necessary for the proofs.
\section{Preliminaries}
We consider the following SDE for $\mathbb{R}^d$-valued stochastic process $X$ with L{\'e}vy noise:
\begin{align}\label{1}
dX(t) = f(t, X(t))dt + g(t, X(t))d L(t),
\end{align}
where $f$ is $\mathbb{R}^d$-valued, $g$ is $\mathbb{R}^{d \times d}$-valued, $L$ is a
two-sided $\mathbb{R}^{d\times d}$-valued L{\'e}vy process(for more details, see Definition \ref{3})
 defined on $(\Omega, \mathcal{F}, \mathbb{P}, (\mathcal{F}_t)_{t \in \mathbb{R}})$, i.e.,
\begin{equation*}
\begin{aligned}
L(t) : =
\begin{cases}
L_1(t), ~~t \geq 0,&\\
-L_2(-t), t < 0,
\end{cases}
\end{aligned}
\end{equation*}
$L_1$ and $L_2$ are two independent, identically distributed L{\'e}vy process, involving $b, Q, W, N$ in Proposition \ref{2}.
 For convenience, we assume that trace $Q < \infty$.
With the assistance of \eqref{4}, we derive
\begin{align*}
\int_{|x| \geq 1} \nu(dx) < \infty.
\end{align*}
Furthermore, we denote $\int_{|x| \geq 1} \nu(dx) < \infty := e$.

Utilizing the L{\'e}vy-It{\^o} composition(Proposition \ref{2}), \eqref{1} can be written by
\begin{align*}
dX(t) &= (f(t, X(t)) + g(t, X(t))b) dt + g(t, X(t))dW(t) +\int_{|x| < 1} g(t, X(t-))x N_1(dt, dx) \\&~~+ \int_{|x| \geq 1}g(t, X(t-)) xN(dt, dx).
\end{align*}
More generally, we consider the following SDE:
\begin{align}\label{6}
dX(t) &= f(t, X(t))dt +g(t, X(t))dW(t)+ \int_{|x| < 1}F(t, X(t-), x)N_1(dt, dx)\notag\\&~~+ \int_{|x|\geq 1}G(t, X(t-), x)N(dt, dx),
\end{align}
where \begin{align*}
&f: \mathbb{R}^+ \times \mathcal{L}^2 (\mathbb{P}, \mathbb{R}^d) \rightarrow \mathcal{L}^2 (\mathbb{P}, \mathbb{R}^d) ,\\&
g: \mathbb{R}^+ \times \mathcal{L}^2 (\mathbb{P}, \mathbb{R}^d) \rightarrow \mathcal{L}^2 (\mathbb{P}, \mathbb{R}^{d \times d}),\\&
F: \mathbb{R}^+ \times \mathcal{L}^2 (\mathbb{P}, \mathbb{R}^d) \times \mathbb{R}^d \rightarrow \mathcal{L}^2 (\mathbb{P}, \mathbb{R}^d),\\&
G: \mathbb{R}^+ \times \mathcal{L}^2 (\mathbb{P}, \mathbb{R}^d)  \times \mathbb{R}^d \rightarrow \mathcal{L}^2 (\mathbb{P}, \mathbb{R}^{d \times d}).
\end{align*}

Throughout the paper, we define $C_b(\mathbb{R}^d)$ as the space of all bounded and continuous functions. Let $C_{b L}(\mathbb{R}^d)$ be the space of all bounded, Lipschitz continuous functions on $\mathbb{R}^d$ endowed with the norm $||\cdot||_{bL}$ and
\begin{align*}
||f||_{bL} := \sup_{x \in \mathbb{R}^d}|f(x)| + \sup_{x_1, x_2 \in \mathbb{R}^d, x_1 \neq x_2}\frac{|f (x_1) -  f(x_2)|}{|x_1 -x_2|},
\end{align*} The scalar product and norm in $\mathbb{R}^d$ are denoted by $\langle \cdot, \cdot \rangle$ and $|\cdot|$ respectively.
As usual, $\mathbb{E}_x$ denotes the expectation with respect to a stochastic process $\{X(t)\}_{t \geq 0}$ when
its initial value is $X(0) = x \in \mathbb{R}^d$. For $p > 0$, $\mathcal{L}^p(\Omega, \mathbb{R}^d)$ denotes the space
 of all $\mathbb{R}^d$-valued random variables $\xi$, such that $\mathbb{E}|\xi|^p = \int_{\Omega} |\xi|^p d\mathbb{P}
  < \infty$. $\mathcal{P}(\mathbb{R}^d)$ is the space of probability measures on $\mathbb{R}^d$, and for any $\mu_1,
  \mu_2 \in \mathcal{P}(\mathbb{R}^d)$, we introduce a Wasserstein metric:
  \begin{align*}
  d_L(\mu_1, \mu_2) = \sup_{Lip(f) \leq 1} \left|\int_{\mathbb{R}^d} f(x) \mu_1(dx) - \int_{\mathbb{R}^d} f(x) \mu_2(dx) \right|,
  \end{align*}
  where $$Lip(f):= \sup_{x_1, x_2 \in \mathbb{R}^d, x_1 \neq x_2} \frac{|f(x_1) - f(x_2)|}{|x_1 - x_2|}.$$
\section{Existence of $\tau$-periodic measure}
We now introduce some definitions that will be used in this section.
\begin{definition}
An $\mathbb{R}^d$-valued stochastic process $X(t)$ is said to be $\tau$-periodic in distribution if its probability
 distribution function $\mu_{X(\cdot)} : \mathbb{R} \rightarrow [0, 1]$ is a $\tau$-periodic function, i.e.,\begin{align*}
\mu_{X(t +\tau)}(A) = \mu_{X(t)}(A), ~~~~\forall t \in \mathbb{R}, A \in \mathcal{B}(\mathbb{R}^d),\end{align*}
where $\mathcal{B}(\mathbb{R}^d)$ is the Borel $\sigma$-algebra on $\mathbb{R}^d$, $\mu_{X(t)} := \mathbb{P} \circ [X(t)]^{-1}$.
\end{definition}

\begin{definition}
Let $X_{\xi}(t)$ be the solution to \eqref{6} with initial value $\xi \in \mathbb{R}^d$, it is said to be a $\tau$-periodic solution in distribution, provided that the conditions below hold:

 (H1) $X_{\xi}(t)$ is $\tau$-periodic in distribution;

 (H2) There exist a stochastic process $\tilde{W}$ with the same distribution as $W$, and $\tilde{N}$ has the same distribution as $N$ with the compensated Poisson random measure $\tilde{N}_1$,
such that $X_{\xi}(t+\tau)$ is a solution of the following:
 \begin{align*}
 dY(t)& = f(t, Y(t)) dt +g(t, Y(t)) d\tilde{W}(t) +\int_{|x| < 1}F(t, Y(t-), x)\tilde{N}_1(dt, dx)\\&~~+\int_{|x|\geq 1} G(t, Y(t-), x) \tilde{N}(dt, dx).  \end{align*}
\end{definition}

\begin{definition}
A sequence of measures $\{\mu_n\} \subset \mathcal{P}(\mathbb{R}^d)$ is said to be weakly convergent to a measure $\mu$, if for any $\phi \in C_b(\mathbb{R}^d)$,
\begin{align*}
\int_{\mathbb{R}^d} \phi(x)\mu_n(dx) \rightarrow \int_{\mathbb{R}^d} \phi(x)\mu(dx), ~~~~n \rightarrow \infty.
\end{align*}
For convenience, we also denote it by $\mu_n \overset{w}{\rightarrow}\mu.$
\end{definition}

\begin{definition}
A sequence of $\mathbb{R}^d$-valued stochastic processes $\{Y_n\}$ is said to be convergent in distribution to an $\mathbb{R}^d$-valued stochastic process $Y$, if for all $t \in \mathbb{R}$,
\begin{align*}
\mu_{Y_n(t)} \overset{w}{\rightarrow}\mu_{Y(t)}.
\end{align*}
For convenience, we also denote it by \begin{align*}
Y_n \overset{\mathcal{D}}{\rightarrow}Y.
\end{align*}
\end{definition}
Now we make some hypothesises to ensure the progress of the subsequent work. We also assume for convenience that the initial time is 0.

(H3) $f, g, F, G$ in \eqref{6} are $\tau$-periodic in $t \in \mathbb{R}$, i.e., for any $x \in \mathbb{R}^d$
\begin{align*}
f(t, x) &= f(t+\tau, x), ~~~~~~~~~g(t, x) = g(t+\tau, x),\\
F(t, x, u) &= F(t+\tau, x, u), G(t, x, u) = G(t+\tau, x, u).
\end{align*}

(H4) There exist a positive constant $M$ such that for all $t \geq 0, 2\leq p \leq 4$,
\begin{align*}
|f(t, 0)|^p \vee |g(t, 0)|^p \vee \int_{|u|< 1} \left|F(t, 0, u)\right|^p \nu(du) \vee \int_{|u| \geq 1} \left|G(t, 0, u)\right|^p\nu(du) \leq M^p.
\end{align*}

(H5) There exist a positive number $L$, such that for any $x_1, x_2 \in \mathbb{R}^d, 2 \leq p \leq 4,$
\begin{align*}
|f(t, x_1) - f(t, x_2)|^p &\vee |g(t, x_1) - g(t, x_2)|^p \vee \int_{|u|\leq 1}|F(t, x_1, u) - F(t,x_2, u)|^p \nu(du)
\\&\vee \int_{|u| \geq 1} |G(t, x_1, u) - G(t, x_2, u)|^p \nu(du) \leq L^p|x_1 - x_2|^p.\end{align*}

Based on (H4) and (H5), the existence and uniqueness of strong solutions for \eqref{6} with initial value $\xi \in
\mathcal{L}^2(\mathbb{P}, \mathbb{R}^d)$ can be established, which will be denoted by $X_{\xi}(t)$, for more details,
we can refer to Theorem $3.1$ in \cite{ref31}.

For $t \in [0, \tau)$, let \begin{align*}(Y^0(t), \tilde{W}^0(t), \tilde{N}_1^0(dt, dx), \tilde{N}^0(dt, dx)) &= (X_{\xi}(t), W(t), N_1(dt, dx),N(dt, dx)),\\
W^1(t) &= W(t+\tau) - W(\tau),\\
 N^1(t, x) &= N(t+\tau, x) - N(\tau, x),\\
 N_1^1(t, x) &= N_1(t+\tau, x) - N_1(\tau, x).\end{align*}
 Then
 \begin{align*}
 &X_{\xi}(t+\tau)\\& =\xi + \int_0^{t+\tau} f(r, X_{\xi}(r)) dr + \int_{0}^{t+\tau} g(r, X_{\xi}(r))dW(r) +\int_{0}^{t+\tau}
 \int_{|x| < 1}F(r, X_{\xi}(r-), x) N_1(dr, dx)\\& ~~~~~~+ \int_0^{t+\tau}\int_{|x|\geq 1}G(r, X_{\xi}(r-), x)N(dr, dx)\\&=
 X_{\xi}(\tau)  +\int_{\tau}^{t+\tau} f(r, X_{\xi}(r))dr +\int_{\tau}^{t+\tau}g(r, X_{\xi}(r))dW(r) +\int_{\tau}^{t+
 \tau} \int_{|x| < 1}F(r, X_{\xi}(r-), x) N_1(dr, dx) \\&~~~~~~+ \int_{\tau}^{t+\tau} \int_{|x| \geq 1}G(r, X_{\xi}(r-), x)N(dr, dx)
 \\& \overset{\mathcal{D}}{=} X_{\xi}(\tau) +\int_0^t f(u+\tau, X_{\xi}(u+\tau)) du + \int_{0}^t g(u+\tau, X_{\xi}(u+\tau)) d(W(u+\tau)- W(\tau))\\&
 ~~~~~~+\int_0^t \int_{|x| < 1}F(u+\tau, X_{\xi}(u+\tau-))d(N_1(u+\tau, x) - N_1(\tau, x))\\&~~~~~~+\int_0^t \int_{|x|\geq 1}G(u+\tau, X_{\xi}(u+\tau-), x)
 d(N(u+\tau, x) - N(\tau, x))\\&
 =  X_{\xi}(\tau) +\int_0^t f(u+\tau, X_{\xi}(u+\tau)) du + \int_{0}^t g(u+\tau, X_{\xi}(u+\tau)) dW^1(u)\\&
 ~~~~~~+\int_0^t \int_{|x| < 1}F(u+\tau, X_{\xi}(u+\tau-))N_1^1(du, dx)+\int_0^t \int_{|x|\geq 1}G(u+\tau, X_{\xi}(u+\tau-), x)
 N^1(du, dx),
 \end{align*}
 where the third equality is achieved in the sense of distribution. Let $Y^1(t) = X_{\xi}(t+\tau)$, then $(Y^1(t), W^1(t), N_1^1(t, x), N^1(t, x))$ is a weak solution to  \eqref{6}.
 By continuing this process, let \begin{align*}Y^k(t) &= X_{\xi}(t+k\tau),\\ W^k(t)&= W(t+k\tau)- W(k\tau),\\
 N_1^k(t, x) &= N_1(t+k\tau, x) - N_1(k\tau, x),\\
 N^k(t, x) &= N(t+k\tau, x) - N(k\tau, x).\end{align*}
 Then $(Y^k(t), W^k(t), N_1^k(t,x), N^k(t, x))_{k \in \mathbb{N}}$ also satisfy \eqref{6}.

The following theorem states that, under appropriate conditions, the solutions of \eqref{6} have finite $p$th moments
 within a finite time interval, where $2 \leq p \leq 4$.
 \begin{theorem}\label{5}
 Suppose that (H4)-(H5) hold. For $2 \leq p \leq 4, \xi \in \mathcal{L}^p(\mathbb{P}, \mathbb{R}^d)$,
 $s \in [0, \tau]$, we have
 $$\mathbb{E}(\sup_{0 \leq t \leq s}|X_{\xi}(t)|^p) \leq  (1+5^{p-1}\mathbb{E}|\xi|^p)e^{as},$$
where $a= 5^{p-1}(L^p +M^p)2^{\frac{p}{2}-1}\left[(1+(2e)^{p-1})\tau^{p-1}+ 2(1+2^{p-2})\left( \frac{p^3}{2(p-1)}\right)^{\frac{p}{2}}\tau^{\frac{p-2}{2}} \right]$.
 \end{theorem}
 \begin{proof}Note that
 \begin{align*}
 &\mathbb{E}(\sup_{0 \leq t\leq s}|X_{\xi}(t)|^p)\\&\leq
 5^{p-1} \mathbb{E}|X_{\xi}(0)|^p +5^{p-1} \mathbb{E}\left(\int_0^s |f(r, X_{\xi}(r))| dr\right)^p +5^{p-1}
  \mathbb{E}\left(\sup_{0\leq t \leq s} |\int_0^t g(r, X_{\xi}(r))dW(r)|^p\right) \\&~~+5^{p-1} \mathbb{E}\left
  (\sup_{0\leq t\leq s}
  |\int_0^t \int_{|x| <1}F(r, X_{\xi}(r-), x) N_1(dr, dx)|^p\right)\\
 &~~+5^{p-1} \mathbb{E}\left(\sup_{0\leq t\leq s}
  |\int_0^t\int_{|x|\geq 1}G(r, X_{\xi}(r-), x)N(dr, dx)|^p\right)
 \\&=:I_1 +I_2+I_3+I_4+I_5.
 \end{align*}
 By Holder's inequality, we have the following estimations:
 \begin{align*}
 I_2 \leq (5\tau)^{p-1}\mathbb{E}\left(\int_0^s |f(r, X_{\xi}(r))|^p dr\right) \leq (5\tau)^{p-1} \left[2^{p-1}\mathbb{E}\int_0^s
 L^p|X_{\xi}(r)|^p dr +2^{p-1}M^p\right].
 \end{align*}
Combining with Lemma \ref{9}, we have
 \begin{align*}
 I_3:&= 5^{p-1} \mathbb{E}\left(\sup_{0\leq t \leq \tau}|\int_0^t g(r, X_{\xi}(r)) dW(r)|^p\right) \leq 5^{p-1}
 \left( \frac{p^3}{2(p-1)}\right)^{\frac{p}{2}} \tau^{\frac{p-2}{2}}\mathbb{E}\int_0^{s}|g(r, X_{\xi}(r))|^p dr\\&
 =5^{p-1} \left( \frac{p^3}{2(p-1)}\right)^{\frac{p}{2}} \tau^{\frac{p-2}{2}}\left[\mathbb{E}\int_0^{s} 2^{p-1}|g(r, X_{\xi}(r))
 - g(r, 0)|^p +2^{p-1}|g(r, 0)|^p dr \right]\\&
\leq5^{p-1}\left( \frac{p^3}{2(p-1)}\right)^{\frac{p}{2}} \tau^{\frac{p-2}{2}}\left[\mathbb{E}\int_0^{s} 2^{p-1}L^p|X_{\xi}(r)|^p
 +2^{p-1}M^p dr  \right]\\&
 \leq 5^{p-1}\left( \frac{p^3}{2(p-1)}\right)^{\frac{p}{2}}2^{p-1} \tau^{\frac{p-2}{2}}\left[\mathbb{E}\int_0^{s}L^p\mathbb{E} |X_{\xi}(r)|^p dr
 +M^p dr  \right]\\&
 \leq 5^{p-1}\left( \frac{p^3}{2(p-1)}\right)^{\frac{p}{2}}2^p \tau^{\frac{p-2}{2}}(L^p+M^p)\left[\mathbb{E}\int_0^{s} |X_{\xi}(r)|^p
 dr +1  \right].\\&
 \end{align*}
Similar to the estimation of $I_3$, with he help of Lemma \ref{10}, we perform the necessary estimates for $I_4$ and $I_5$:
 \begin{align*}
 I_4:&= 5^{p-1}\sup_{0\leq t \leq  s} \left|\int_0^t \int_{|x| < 1}F(r, X_{\xi}(r-), x) N_1(dr, dx)\right|^p \\&
 \leq 5^{p-1} \left( \frac{p^3}{2(p-1)}\right)^{\frac{p}{2}} \tau^{\frac{p-2}{2}} \mathbb{E}\left(\int_0^s\int_{|x| < 1}
 |F(r, X_{\xi}(r-), x) - F(r, 0, x)|^p\nu(dx)dr\right)\\&
 \leq 5^{p-1} \left( \frac{p^3}{2(p-1)}\right)^{\frac{p}{2}} \tau^{\frac{p-2}{2}} \mathbb{E}\left(\int_0^s\int_{|x|< 1}
 2^{p-1}|F(r, X_{\xi}(r-), x) - F(r, 0, x)|^p \nu(dx)dr  +2^{p-1}M^p \tau \right)\\&
 \leq  5^{p-1} \left( \frac{p^3}{2(p-1)}\right)^{\frac{p}{2}} \tau^{\frac{p-2}{2}} \mathbb{E}\left(\int_0^s\int_{|x|< 1}
 2^{p-1}L^p |X_{\xi}(r)|^p \nu(dx)dr  +2^{p-1}M^p \tau \right),
 \end{align*}

 \begin{align*}
 I_5:&=5^{p-1}\mathbb{E}\left(\sup_{0 \leq t \leq s}\left|\int_0^t \int_{|x|\geq 1}G(r, X_{\xi}(r-), x) N(dr, dx)\right|^p\right)\\&
 \leq 5^{p-1} \mathbb{E}\left(2^{p-1}\left| \int_0^t \int_{|x|\geq 1} G(r, X_{\xi}(r-), x)N_1(dr, dx)\right|^p +2^{p-1}\left| \int_0^t \int_{|x|\geq 1} G(r, X_{\xi}(r-), x) \mu(dx)dr\right|^p\right)\\&
 \leq 10^{p-1}\left( \frac{p^3}{2(p-1)}   \right)\tau^{\frac{p-2}{2}}\mathbb{E}\left(\int_0^t\int_{|x|\geq 1}|G(r, X_{\xi}(r-), x)|^p \nu(dx) dr\right)
 \\&~~+ (10e\tau)^{p-1} \mathbb{E}\left(\int_0^t\int_{|x|\geq 1}|G(r, X_{\xi}(r-), x)|^p \nu(dx)dr   \right)\\&\leq
 10^{p-1}\left( \frac{p^3}{2(p-1)}   \right)\tau^{\frac{p-2}{2}}\mathbb{E}\left(\int_0^t\int_{|x|\geq 1}2^{p-1}\mathbb{E}|X_{\xi}(r)|^p L^p \nu(dx) dr + 2^{p-1}M^p\tau\right)\\&~~+ (10e\tau)^{p-1}
 \mathbb{E}\left(\int_0^t\int_{|x|\geq 1}2^{p-1}\mathbb{E}|X_{\xi}(r)|^p L^p \nu(dx) dr + 2^{p-1}M^p\tau \right).
 \end{align*}
 Hence
 \begin{align*}
 1+\mathbb{E}(\sup_{0\leq t \leq s}|X_{\xi}(t)|^p) \leq 1+ 5^{p-1}\mathbb{E}|\xi|^p +a\int_0^t[1+\mathbb{E}(\sup_{0\leq u \leq r}|X_{\xi}(u)|^p)]dr,
 \end{align*}
 where $a= 5^{p-1}(L^p +M^p)2^{\frac{p}{2}-1}\left[(1+(2e)^{p-1})\tau^{p-1}+ 2(1+2^{p-2})\left( \frac{p^3}{2(p-1)}\right)^{\frac{p}{2}}\tau^{\frac{p-2}{2}} \right]$.

 Then it follows from Gronwall's inequality that
 \begin{align*}
 1+\mathbb{E}(\sup_{0\leq t\leq s}|X_{\xi}(t)|^p) \leq (1+5^{p-1}\mathbb{E}|\xi|^p)e^{as}
 \end{align*}
 for $s \in [0, \tau]$.

 Then  \begin{align*}
\mathbb{E}(\sup_{0\leq t\leq s}|X_{\xi}(t)|^p) \leq (1+5^{p-1}\mathbb{E}|\xi|^p)e^{as}
 \end{align*}
 for $s \in [0, \tau]$.
 \end{proof}
 \begin{remark}
 When $0 < p <2$, from Holder's inequality, it holds that
 \begin{align*}
 \mathbb{E}|X_{\xi}(t)|^p \leq (\mathbb{E}|X_{\xi}(t)|^2)^{\frac{p}{2}} \leq (1+5\mathbb{E}|\xi|^2)^{\frac{p}{2}} e^{\frac{pat}{2}},
 \end{align*}
 where $a$ is from Theorem \ref{5}.
 \end{remark}

 In the subsequent discussion, we will establish the existence of the $\tau$-periodic measure for
 \eqref{6}. In the process of proof, we will rely on certain facts from Skorokhod theorem (\cite{ref42}).
 For ease of presentation,
 we will give the existed results in the appendix, specifically Lemma \ref{11} and Lemma \ref{12}. Before stating our theorem,
 we introduce another hypothesises:

 (H6) For any $t \in [0, \tau), k\in \mathbb{N}$, $\mu_{Y^k(t)}= \mu_{X_{\xi}(t+k\tau)}:= \mathbb{P}\circ Y^k(t)^{-1}$, satisfying
 \begin{align}\label{13}
 \lim_{k \rightarrow \infty} \frac{1}{n_k+1} \sum_{N=0}^{n_k}d_L(\mu_{X_{\xi}(t+(N+1)\tau)}, \mu_{X_{\xi}(t+N\tau)}) = 0,
 \end{align}
 where $\{n_k\}$ is a sequence of integers tending to $+ \infty$.

(H7) For $2 \leq p \leq 4$, $\{X_{\xi}(k\tau)\}_{k \in \mathbb{N}}$ is uniformly bounded, i.e., there exist a positive constant
$C > 0$ such that
$$
\mathbb{E}|X_{\xi}(k\tau)|^p \leq C
$$
for any $k \in \mathbb{N}$.
 \begin{theorem}
 Suppose that (H3)-(H7) are achieved. Then \eqref{6} has a $\tau$-periodic measure.
 \end{theorem}
 \begin{proof}
 For $t \in [0, \tau), k \in \mathbb{N}$, recall the construction of $(Y^k(t), W^k(t), N_1^k(t,x), N^k(t, x))_{k \in \mathbb{N}}$.
 Define a random variable $\eta_k$, where $\mathbb{P}(\eta_k = N)
 = \frac{1}{k+1}, N= 0, 1, \cdots, k$, and $\eta_k$ is independent of $W$, $N_1$ and $\xi$, then
  $(Y^{\eta_k}(t), W^{\eta_k}(t), N_1^{\eta_k}(t,x), N^{\eta_k}(t, x))_{k \in \mathbb{N}}$ is
  a solution to \eqref{6}.
 For any $A \in \mathcal{B}(\mathbb{R}^d)$, $k \in \mathbb{N}$,
 \begin{align*}
 \mathbb{P}(Y^{\eta_k}(t) \in A) = \frac{1}{k+1} \sum_{N=0}^k \mathbb{P}(X_{\xi}(t+N\tau) \in A).
 \end{align*}
 Then combining (H7) and Chebyshev's inequality, we obtain a uniform bound on $k$, that is
 \begin{align*}
 \mathbb{P}(|Y^{\eta_k}(0)| > R) &= \frac{1}{k+1} \sum_{N=0}^k \mathbb{P}(|X_{\xi}(N \tau)| > R) \\&
 \leq \frac{1}{k+1} \sum_{N=0}^k \frac{\mathbb{E}|X_{\xi}(N\tau)|^2}{R^2} \rightarrow 0, ~~~~R\rightarrow \infty,
 \end{align*}
 which satisfies conditions in Lemma \ref{11}, Lemma \ref{12}, and Lemma \ref{39}. Indeed, there is another probability space
 $(\tilde{\Omega}, \tilde{\mathcal{F}}, \tilde{P})$, and there is a sequence $\tilde{Y}^{\eta_k}(0)(k \in \mathbb{N})$ in which,
 with the same distribution as $Y^{\eta_k}(0)$, and there also exists a subsequence $\tilde{Y}^{\eta_{n_k}}(0)$, which converges to
 $\tilde{Y}(0)$ in probability. For $\tilde{Y}(0)$ and $\tilde{Y}^{\eta_{n_k}}(0)$, we can get random variables in the original
 probability space $(\Omega, \mathcal{F}, \mathbb{P})$ with the same distribution as them respectively,
 which will be denoted by $Y(0)$ and $Y^{\eta_{n_k}}(0)$. Reusing (H7), for $2 \leq p \leq 4$, we have
\begin{align*}
\tilde{ \mathbb{E}}|\tilde{Y}^{\eta_{n_k}}(0)|^p = \mathbb{E}|Y^{\eta_{n_k}}(0)|^p \leq C < \infty.
 \end{align*}
 Thanks to Lemma \ref{12} and Remark \ref{15}, there exists a $\delta > 0$ such that for any $A \in \tilde{\mathcal{F}}$ with
 $\tilde{\mathbb{P}}(A) \leq \delta$, we have
$$
\sup_{\tilde{Y}^{\eta_{n_k}}(0) \in \mathcal{A}}\int_{A}|\tilde{Y}^{\eta_{n_k}}(0)|^2 d \tilde{\mathbb{P}} \leq \epsilon.$$
 Then applying the well-known Vitali's convergence theorem and some corollary of Lebesgue's dominated convergence
 theorem, such as Lemma \ref{16}, we have
 \begin{align*}
\tilde{\mathbb{E}}|\tilde{Y}^{\eta_{n_k}}(0) - \tilde{Y}(0)|^2 \rightarrow 0, ~~~~k\rightarrow \infty.
\end{align*}
   Let $(\tilde{Y}^{\eta_{n_k}}(t), W^{\eta_{n_k}}(t), N_1^{\eta_{n_k}}(t, x), N^{\eta_{n_k}}(t, x))$ be a weak solution to \eqref{6} with initial
   condition $\tilde{Y}_{n_k}(0)$ on the probability space $(\tilde{\Omega}, \tilde{\mathcal{F}}, \tilde{\mathcal{P}})$.
  It follows from Cauchy-Schwarz's inequality, Lemma \ref{9} that
  \begin{align*}
\tilde{\mathbb{E}}|\tilde{Y}^{\eta_{n_k}}(t) - \tilde{Y}(t)|^2  \rightarrow 0, ~~~~n_k \rightarrow \infty.
\end{align*}
Indeed,
\begin{align*}
&\mathbb{E}|\tilde{Y}^{\eta_{n_k}}(t) - \tilde{Y}(t)|^2 \\&\leq 5 \mathbb{E}|\tilde{Y}^{\eta_{n_k}}(0) - \tilde{Y}(0)|^2
+ 5\mathbb{E}\left|\int_0^t f(r, \tilde{Y}^{\eta_{n_k}}(r))- f(r, \tilde{Y}(r)) dr\right|^2 \\&~~+ 5 \mathbb{E}\left|\int_0^t
g(r, \tilde{Y}^{\eta_{n_k}}(r))- g(r, \tilde{Y}(r))d W^{\eta_{n_k}}(r) \right|^2\\&~~+ 5\mathbb{E}\left|\int_0^t \int_{|x|< 1}
F(r, \tilde{Y}^{\eta_{n_k}}(r), x) - F(r, \tilde{Y}(r), x) N_1(dr, dx) \right|^2 \\&~~+5\mathbb{E}|\int_0^t
\int_{|x| \geq 1}G(r, \tilde{Y}^{\eta_{n_k}}(r), x)- G(r, \tilde{Y}(r), x) N_1(dr, dx)\\&~~+ \int_0^t  \int_{|x| \geq 1}G(r,\tilde{Y}^{\eta_{n_k}}(r), x)
- G(r, \tilde{Y}(r), x) \nu^1(dx)dr |^2\\&
\leq 5\mathbb{E}|\tilde{Y}^{\eta_{n_k}}(0) - \tilde{Y}(0)|^2 +5t \mathbb{E}\int_0^t |f(r, \tilde{Y}^{\eta_{n_k}}(r))- f(r, \tilde{Y}(r))|^2 dr
\\&~~+ 5 \mathbb{E}\int_0^t|g(r, \tilde{Y}^{\eta_{n_k}}(r))- g(r, \tilde{Y}(r))|^2dr\\&~~+ 5\mathbb{E}\int_0^t \int_{|x|< 1}
|F(r, \tilde{Y}^{\eta_{n_k}}(r), x) - F(r, \tilde{Y}(r), x)|^2 \nu^1(dx)dr \\&~~+10\mathbb{E}\int_0^t
\int_{|x| \geq 1}|G(r, \tilde{Y}^{\eta_{n_k}}(r), x)- G(r, \tilde{Y}(r), x)|^2\nu^1(dx)dr\\&~~+
10\mathbb{E}\left[\int_0^t \int_{|x|\geq 1}\nu^1(dx)dr \int_0^t \int_{|x|\geq 1}|G(r, Y^{\eta_{n_k}}(r), x)
 - G(r, \tilde{Y}(r), x)|^2\nu^1(dx)dr\right]\\&
 = 5\mathbb{E}|\tilde{Y}^{\eta_{n_k}}(0) - \tilde{Y}(0)|^2 +(20tL^2 +10teL^2) \mathbb{E}\int_0^t |\tilde{Y}^{\eta_{n_k}}(r) - \tilde{Y}(r)|^2 dr.
\end{align*}
Applying Gronwall's inequality, we have
\begin{align*}
\mathbb{E}|\tilde{Y}^{\eta_{n_k}}(t) - \tilde{Y}(t)|^2 \leq 5\mathbb{E}|\tilde{Y}^{\eta_{n_k}}(0) - \tilde{Y}(0)|^2
 e^{(20tL^2 +10teL^2)t} \rightarrow 0,~~~~k \rightarrow \infty.
\end{align*}
Then we have
\begin{align*}
\mathbb{P} \circ Y^{\eta_{n_k}}(t)^{-1} = \tilde{\mathbb{P}} \circ \tilde{Y}^{\eta_{n_k}}(t)^{-1} \rightarrow \tilde{\mathbb{P}}\circ
\tilde{Y}(t)^{-1}
\end{align*}
uniformly on $[0, \tau]$. For $Y(0)$ is identically distributed with $\tilde{Y}(0)$, $Y(t)$ has the same distribution with
$\tilde{Y}(t)$.  Indeed, we will verify $Y(\tau) \overset{d}{=} Y(0)$.
  \begin{align*}
  d_{L}(\mathbb{P} \circ Y(\tau)^{-1}, \mathbb{P}\circ Y(0)^{-1}) &
 = \lim_{k \rightarrow \infty} d_L(\tilde{\mathbb{P}}\circ \tilde{Y}^{\eta_{n_k}}(\tau)^{-1}, \tilde{\mathbb{P}} \circ \tilde{Y}^{\eta_{n_k}}(0)^{-1})\\&
 = \lim_{k \rightarrow \infty} d_L(\mathbb{P}\circ Y^{\eta_{n_k}}(\tau)^{-1}, \mathbb{P} \circ Y^{\eta_{n_k}}(0)^{-1})\\&
 =\lim_{k \rightarrow \infty} \sup_{Lip(\Phi) \leq 1} \left| \int_{\Omega} \Phi(Y^{\eta_{n_k}}(\tau)) - \Phi(Y^{\eta_{n_k}}(0))
 d \mathbb{P}\right|\\&
 =\lim_{k \rightarrow \infty} \sup_{Lip(\Phi) \leq 1}\frac{1}{n_k +1} \sum_{N=0}^{n_k} \left|\int_{\Omega}
 \Phi(X_{\xi}(\tau+ N\tau)) - \Phi(X_{\xi}(N\tau))d \mathbb{P} \right|\\&
 \leq \lim_{k \rightarrow \infty} \frac{1}{n_k +1} \sum_{N=0}^{n_k} d_L(\mu_{X_{\xi}(N\tau +\tau)}, \mu_{X_{\xi}(N\tau)}) = 0.
  \end{align*}
  So $Y(t)$ is the unique $\tau$-periodic solution in distribution  to \eqref{6}, and $\mu_{Y(t)} := \mathbb{P}\circ Y(t)^{-1}$
  is a $\tau$-periodic measure, in other words, $\mu_{Y(t)} = \mu_{Y(t+\tau)}$.
 \end{proof}
 \begin{remark}
Due to the independent increment property of L{\'e}vy process, we know that $\{Y(t)\}_{t \geq 0}$ is a Markov process, satisfying the crucial Markov property: the future behaviour of the process,
 given what has occerred up to time $s$, is
the same as the behaviour obtained when starting the process at $Y(s)$.
\end{remark}
  Now we introduce some notation commonly used in the context of Markov process:
  let $\mathbb{E}_x$ denote the expectation with respect to the stochastic process $\{Y(t)\}_{t \geq 0}$
  when its initial value is $Y(0) = x \in \mathbb{R}^d$, let $\mathbb{E}_{s,Y(s)}$ denote the expectation of $Y(t)$
   when its initial time and initial value are $(s, Y(s)) \in \mathbb{R} \times \mathbb{R}^d$.
   So $\mathbb{E}[\Phi(Y(t)) \mid \mathcal{F}_{N\tau}] = \mathbb{E}_{N\tau, Y(N\tau)} [\Phi(Y(t))]$.
   Furthermore, from the uniqueness of weak solutions, i.e., $Y(t+N\tau) \overset{\mathcal{D}}{=} Y_{Y(N\tau)}(t)$, we have
   $\mathbb{E}[\Phi(Y(t)) \mid \mathcal{F}_{N\tau}] = \mathbb{E}_{Y(N\tau)}[\Phi(Y(t-N\tau))]$.

    For $\Phi \in C_b(\mathbb{R}^d)$, let $P_{s, t}\Phi(Y(s)) = \mathbb{E}[\Phi(Y(t)) \mid \mathcal{F}_s],$ and let
   $P_{s, t}^*$ denote the dual operator
  corresponding to $P_{s, t}$. Then by the Markov property, we have $P_{s, t}^* \mu_{Y(s)} = \mu_{Y(t)}$.

\begin{remark}\label{35}
For $x_1 \in \mathbb{R}^d$, we have $P_{0, \tau}^* \delta_{x_1} = \delta_{X_{x_1}(\tau)}$, where $\delta_x$ denotes the dirac measure at $x \in \mathbb{R}^d$.\begin{proof}In fact,
 \begin{align*}
 d_{L}(P_{0, \tau}^* \delta_{x_1}, \delta_{X_{x_1}(\tau)})&= \sup_{Lip(f) \leq 1}\left|\int_{\mathbb{R}^d} f  P_{0, \tau}^* \delta_{x_1}(dx) - \mathbb{E} \int_{\mathbb{R}^d} f \delta_{X_{x_1}(\tau)}(dx)\right|
 \\& = \sup_{Lip(f) \leq 1}|\mathbb{E}[f(X(\tau))\mid X(0)= x_1] - \mathbb{E}f(X_{x_1}(\tau))|\\&
 =\sup_{Lip(f) \leq 1}|\mathbb{E}[f(X_{x_1}(\tau)) - f(X_{x_1}(\tau))] |\\&
 =0.
 \end{align*}
 \end{proof}
 \end{remark}
 \begin{remark}\label{36}
 For any $x_1, x_2 \in \mathbb{R}^d$,
 \begin{align*}
 d_{L}(\delta_{x_1}, \delta_{x_2})= \sup_{Lip(f) \leq 1}\left|f(x_1) - f(x_2)\right|= |x_1 - x_2|.
 \end{align*}
 \end{remark}

We also make the following assumptions.

 (H8) For $\gamma \in (0, 1)$, there exists a continuous function $a: \mathbb{R}^+ \rightarrow \mathbb{R}^+ \setminus \{0\}$
 with $$r:= \lim_{k \rightarrow \infty} a(k \tau) < 1,$$ such that for any $x_1, x_2 \in \mathbb{R}^d, t \geq 0$,
 \begin{align*}
 d_{L}(P_{0, t}^*\delta_{x_1}, P_{0, t}^* \delta_{x_2}) \leq a(t)|x_1 -x_2|.
 \end{align*}

 (H9) Suppose that there exists a $\lambda \in( L+8M^2 +\frac{1}{8}, L+8M^2 +\frac{1}{8} + \frac{1}{4}\log 2)$, such that
 \begin{align*}
2 x\cdot f(r, x) +\int_{|u| \geq 1}2x \cdot G(r, x, u)\nu(du) \leq -\lambda(1+|x|^2).
 \end{align*}

 Now we have the following.

 \begin{theorem}\label{21}
 Suppose that (H3)-(H8) hold. Then there exists a $C > 0, \gamma \in (0, 1)$ such that for any $\mu_1, \mu_2 \in
 \mathcal{P}(\mathbb{R}^d), t \geq 0$,
 \begin{align*}
 d_{L}(P_{0, t}^* \mu_1, P_{0, t}^* \mu_2) \leq Ce^{-\gamma t} d_L(\mu_1, \mu_2).
 \end{align*}
 \end{theorem}
  \begin{proof}
  By the definition of $r$, there exists an $N \in \mathbb{N}$ and $a \in (r, 1)$, such that for all $k \geq N$, we
  have $a(k\tau) \leq \alpha$.
 Now for any $x_1, x_2 \in \mathbb{R}^d, 0 \leq t \leq N\tau, t = k\tau +s$ for a unique $0 \leq k \leq N$ and $s \in [0, \tau)$,
 combining with Remark \ref{35}, \begin{align}\label{37}
 d_L(P_{0, t}^* \delta_{x_1}, P_{0, t}^* \delta_{x_2})= d_L(P_{k\tau, t}^*P_{(k-1)\tau, k\tau}^*\cdots P_{0, \tau}^*\delta_{x_1},
 P_{k\tau, t}^*P_{(k-1)\tau, k\tau}^*\cdots P_{0, \tau}^*\delta_{x_2}).
 \end{align}
 For $a(t)$ is a continuous function, there exists an $M_1 > 0$, such that $|a(t)| \leq M_1$ for $t \in [0, N\tau]$, and
 \begin{align}\label{38}
 d_{L}(P_{k\tau, t}\delta_{X_{x_1}(k\tau)}, P_{k\tau}^* \delta_{X_{x_2}(k\tau)})\notag & ~~\leq M_1
 |X_{x_1}(k\tau) - X_{x_2}(k\tau)|= M_1 d_{L}(\delta_{X_{x_1}(k\tau)}, \delta_{X_{x_2}(k\tau)}) \notag\\&~~= M_1d_{L}
 (P_{0, \tau}^* \delta_{X_{x_1}((k-1)\tau)}, P_{0, \tau}^* \delta_{X_{x_2}((k-1)\tau)})\notag\\&~~\leq M_1^2 d_{L}(\delta_{X_{x_1}((k-1)\tau)}, \delta_{X_{x_2}(k\tau)})\notag\\&~~\leq \cdots \leq M_1^k d_L(\delta_{X_{x_1}(\tau)}, \delta_{X_{x_2}(\tau)}) \leq M_1^{k+1}|x_1 - x_2|,
 \end{align}
 where we use the fact from Remark \ref{36}.
Considering \eqref{37} and \eqref{38}, $$d_{L}(P_{0, t}^* \delta_{x_1}, P_{0, t}^*\delta_{x_2}) \leq M^{k+1}|x_1 - x_2|.$$
 By choosing $\tilde{C}= M^{k+1} C^{N\tau}$, we have
 \begin{align*}
 d_{L}(P_{0, t}^* \delta_{x_1}, P_{0, t}^* \delta_{x_2}) \leq \tilde{C} e^{-N\tau} |x_1 - x_2| \leq \tilde{C} e^{-t}
 |x_1 - x_2|.
 \end{align*}
 For $t > N\tau$, one has $t = k N\tau +\beta, k \in \mathbb{N}$ and $0 \leq \beta \leq N\tau$. Combing (H8) and the
 definition of $r$, we have
 \begin{align*}
 d_{L}(P_{0, kN\tau}^*\delta_{x_1}, P_{0, kN\tau}^* \delta_{x_2}) \leq \alpha^k |x_1 -x_2|,
 \end{align*}
 then for an appropriate $\gamma \in (0, 1)$,
 \begin{align*}
 d_{L}(P_{0, t}^*\delta_{x_1}, P_{0, t}^*\delta_{x_2}) &= d_{L}(P_{kN\tau, kN\tau+\beta}^* P_{0, kN\tau}^* \delta_{x_1},
 P_{kN\tau, kN\tau+\beta}^*P_{0, kN\tau}^*\delta_{x_2}) \\&\leq \tilde{C}e^{-\beta} d_L(P_{0, kN\tau}^* \delta_{x_1},
 P_{0, kN\tau}^*\delta_{x_2}) \leq \tilde{C} e^{-\gamma \beta} \alpha^{\frac{t-\beta}{N\tau}}|x_1 - x_2| \leq
 Ce^{-\gamma t}|x_1 -x_2|
 \end{align*}
 for some constants $C, \gamma > 0$. Combing with Lemma \ref{18}, we have the desired result.
 \end{proof}
\section{Strong law of large numbers}
 Now we consider a special case of equation \eqref{6}, i.e.,
 \begin{align}\label{17}
 dX(t) = f(t, X(t))dt +g(t) d \tilde{W}(t) +\int_{|x| < 1} F(t, X(t-), x)\tilde{N}_1(dt, dx) +\int_{|x|\geq 1}
 G(t, X(t-), x)\tilde{N}(dt, dx).
 \end{align}
 The coefficients in this equation satisfy (H3)-(H9), thus \eqref{17} has a $\tau$-periodic solution in distribution,
 which will be denoted by $X(t)$ in the following, and $X_{\xi}(t)$ denote the $\tau$-periodic solution to equation \eqref{17}
  with initial value $\xi \in \mathbb{R}^d$.

  For $\gamma \in (0, 1]$, let $C_{bL}^{\gamma}(\mathbb{R}^d)$ be the space of continuous bounded function with finite norms weighted
 by the Lyapunov function $e^{|x|^2}$:
 \begin{align*}
 C_{bL}^{\gamma}(\mathbb{R}^d):= \{\phi \in C_{bL}(\mathbb{R}^d): ||\phi||_{bL, \gamma} < \infty\},
 \end{align*}
 where \begin{align*}
 ||\phi||_{bL, \gamma} := \sup_{x \in \mathbb{R}^d} \frac{|\phi(x)|}{e^{|x|^2}}+\sup_{0 < |x_1 -x_2| \leq 1}
 \frac{|\phi(x_1)-\phi(x_2)|}{|x_1 - x_2|(e^{|x_1|^2} + e^{|x_2|^2})}.
 \end{align*}
 For $\Phi \in C_{bL}^{\gamma}(\mathbb{R}^d)$, define $$\tilde{\Phi}(X_{\xi}(t)) = \Phi(X_{\xi}(t)) - \int_{\mathbb{R}^d} \Phi(x) \mu^*(dx),$$where
$$\mu^* = \frac{1}{\tau} \int_0^{\tau} \mu_{X_{\xi}(t)} dt.$$

Before proving the key theorem of this section, namely the strong law of large numbers, we first make some preparations.
 \begin{lemma}\label{22}
For $t \geq 0$, $\lambda \in( L+8M^2 +\frac{1}{8}, L+8M^2 +\frac{1}{8} + \frac{1}{4}\log 2)$, $4 < \eta_0 < 8$, $\eta \in (0, \eta_0], X(0) = \xi \in \mathbb{R}^d$, we have
\begin{align*}
\mathbb{E}e^{\eta|X(t)|^2}  \leq Ce^{\eta e^{-\alpha t}|\xi|^2} e^{\frac{\eta(6M^2-4\lambda)}{\alpha}} < Ce^{\eta e^{-\alpha t}|\xi|^2},
\end{align*}
where $\alpha = 4\lambda - 4L-32 M^2 - \frac{1}{2} \in (0, \log 2)$.
\end{lemma}
\begin{proof}
For $\alpha = 4\lambda - 4L-32 M^2 - \frac{1}{2} > 0$, applying It{\^o} formula to $e^{\alpha t}|X(t)|^2$ yields
\begin{align*}
e^{\alpha t}|X(t)|^2&= |\xi|^2 +\int_0^t e^{\alpha r} (\alpha |X(r)|^2 +2 X(r)\cdot f(r, X(r)) +|g(r)|^2 +\int_{|u| < 1}
|X(r) +F(r, X(r), u)|^2 \\&~~- |X(r)|^2 -2X(r)\cdot F(r, X(r), u)\nu(du) + \int_{|u| \geq 1}|X(r)+ G(r, X(r), u)|^2 \\&~~-|X(r)|^2\nu(du)) dr+
\int_{0}^t 2e^{\alpha r} X(r)\cdot g(r)d \tilde{W}(r)\\&
=|\xi|^2 +\int_0^t e^{\alpha r} (\alpha |X(r)|^2 +2 X(r)\cdot f(r, X(r)) +|g(r)|^2 +\int_{|u| < 1}
|F(r, X(r), u)|^2\nu(du) \\&~~+\int_{|u| \geq 1}2X(r)\cdot G(r, X(r), u) +|G(r, X(r), u)|^2 \nu(du))dr
+ \int_{0}^t 2e^{\alpha r} X(r)\cdot g(r)d \tilde{W}(r).
\end{align*}
Suppose that $M(t) = \int_0^t 2 X(r) \cdot g(r) d\tilde{W}(r)$, $[M]_t = 4\int_0^t |X(r) \cdot g(r)|^2 dr.$
 Then \begin{align*}
 &\int_0^t e^{\alpha(r-t)} dM(r) - 8\int_0^t e^{\alpha(r-t)} d[M]_r\\&
 =|X(t)|^2 -e^{-\alpha t}|\xi|^2 -\int_0^t e^{\alpha (r-t)} (\alpha |X(r)|^2  + 2X(r) \cdot f(r, X(r)) +|g(r)|^2
 \\&~~+\int_{|u| < 1} |F(r, X(r), u)|^2 \nu(du) + \int_{|u| \geq 1}2 X(r) \cdot G(r, X(r), u) + |G(r, X(r), u)|^2 \nu(du))dr\\&~~
-8\int_0^t e^{\alpha (r-t)} 4|X(r) \cdot g(r)|^2 dr\\&\geq
|X(t)|^2 -e^{-\alpha t }|\xi|^2 -\int_0^t e^{\alpha(r-t)} (\alpha |X(r)|^2 - 4\lambda (1+|X(r)|^2)+6M^2+4L|X(r)|^2 +32M^2|X(r)|^2)dr\\&\geq
|X(t)|^2-e^{-\alpha t }|\xi|^2 -\int_0^t e^{\alpha(r-t)} ((\alpha - 4\lambda+4L+32M^2)|X(r)|^2 +(6M^2 - 4\lambda))dr\\&
\geq |X(t)|^2-e^{-\alpha t }|\xi|^2 -\int_0^t e^{\alpha(r-t)}(6M^2 - 4\lambda)dr\\&
=|X(t)|^2 -e^{-\alpha t }|\xi|^2 - \frac{(6M^2 - 4\lambda)(1-e^{-\alpha t})}{\alpha}\\&
\geq |X(t)|^2 -e^{-\alpha t }|\xi|^2 - \frac{6M^2 - 4\lambda}{\alpha}.
 \end{align*}
 Then for any $K > 0$, we have some facts similar to Lemma A.1 in \cite{ref50}:
 \begin{align*}
& \mathbb{P}\left( |X(t)|^2 -e^{-\alpha t}|\xi|^2 - \frac{6M^2 - 4\lambda}{\alpha} > \frac{K e^{\alpha}}{16}\right)\\&\leq
  \mathbb{P}\left( \int_0^t e^{\alpha(r-t)} dM(r) - 8\int_0^t e^{\alpha(r-t)} d[M]_r > \frac{K e^{\alpha}}{16}\right) \leq e^{-K},
 \end{align*}
 which is equivalent to
 \begin{align*}
 \mathbb{P}\left(e^{ \eta_0(|X(t)|^2 -e^{-\alpha t}|\xi|^2 - \frac{6M^2 - 4\lambda}{\alpha})} > e^{\frac{K e^{\alpha} \eta_0}{16}}\right) \leq e^{-K}.
 \end{align*}
 We know that for any $c >1$, if a random variable $X$ satisfies $$
 \mathbb{P}(X \geq C) \leq \frac{1}{C^c}$$ for every $C \geq 1$, then
 \begin{align*}
 \mathbb{E}X &= \int_{\Omega} X d\mathbb{P} \leq \int_{0 \leq X \leq 1} X d\mathbb{P} + \int_{X \geq 1} X d\mathbb{P}
 \leq 1+ \sum_{n=0}^{\infty} \int_{2^n \leq X \leq 2^{n+1}} X d\mathbb{P} \\&\leq 1+\sum_{n=0}^{\infty}2^{n+1}\frac{1}{2^{cn}}
  \leq \frac{4}{1-2^{1-c}}.
 \end{align*}
 Then let $c = \frac{16}{e^{\alpha}\eta_0 }> 1, C= e^{\frac{K e^{\alpha} \eta_0}{16}} > 1$,
 \begin{align*}
 \mathbb{E}e^{ \eta_0(|X(t)|^2 -e^{-\alpha t}|\xi|^2 - \frac{6M^2 - 4\lambda}{\alpha})} \leq \frac{4}{1 - 2^{1-\frac{16}{e^{\alpha}\eta_0 }}}.
 \end{align*}
 By Holder's inequality, for any $\eta \in (0, \eta_0]$,
 \begin{align*}
 \mathbb{E}e^{ \eta(|X(t)|^2 -e^{-\alpha t}|\xi|^2 - \frac{6M^2 - 4\lambda}{\alpha})} \leq
 \left(\mathbb{E} e^{\eta_0 (|X(t)|^2-e^{-\alpha t}|\xi|^2 - \frac{6M^2 - 4\lambda}{\alpha})}    \right)^{\frac{\eta}{\eta_0}}
 \leq \frac{4}{1-2^{1-\frac{16}{e^{\alpha}\eta_0 }}}.
 \end{align*}
 Then \begin{align*}
 \mathbb{E}e^{ \eta|X(t)|^2} \leq \frac{4}{1- 2^{1-\frac{16}{e^{\alpha}\eta_0 }}} \left[e^{\eta e^{-\alpha t}|\xi|^2 +
 \frac{\eta(6M^2-4\lambda)}{\alpha}} \right] \leq \frac{4}{1- 2^{1- \frac{16}{e^{\alpha}\eta_0 }}} e^{\eta e^{-\alpha t}|\xi|^2}.
 \end{align*}
\end{proof}

Combining with the above result, we make some estimation for the orthogonalized observation function.
 \begin{theorem}\label{20}
 For $\Phi \in C_{bL}^{\gamma}(\mathbb{R}^d), t \geq N\tau, N \in \mathbb{N}$, $\mathbb{E}|\xi|^2 < \infty,$
 \begin{align*}
 |\mathbb{E}[\Phi(X_{\xi}(t))\mid \mathcal{F}_{N\tau}] - \langle \mu^*, \Phi(\cdot) \rangle| \leq  C\|\Phi\|_{bL, \gamma}e^{2|\xi|^2} e^{-\frac{\gamma (t-N\tau)}{{5}}}.
 \end{align*}
 \end{theorem}
 \begin{proof}
 For any $R > 0$, let $\chi_{R}: \mathbb{R}^d \rightarrow \mathbb{R}$ satisfy $0 \leq \chi_R \leq 1$ with $\chi_R(x) =1$
 for $|x| \leq R$, and $\chi_{R}(x) = 0$ for $|x| \geq R+1$. We can choose a $\chi_R$ such that
  $||\chi_{R}||_{bL, \gamma} \leq 2.$

  Without loss of generality that $||\Phi||_{bL, \gamma} \leq 1$. Let $\bar{\chi}_{R} = 1- \chi_{R}$, then
  \begin{align*}
  &|\mathbb{E}[\Phi(X_{\xi}(t))]\mid \mathcal{F}_{N\tau} | - \langle \mu^*, \Phi(\cdot)\rangle|\\&
 = |\mathbb{E}_{X_{\xi}(N\tau)}[\Phi(X_{\xi}(t - N\tau))] - \langle\mu^*, \Phi(\cdot) \rangle|\\&
=|\mathbb{E}_{X_{\xi}(N\tau)}[(\chi_R\Phi) (X_{\xi}(t - N\tau))] + \mathbb{E}_{X_{\xi}(N\tau)}
  [(\bar{\chi}_{R}\Phi)(X_{\xi}(t-N\tau))] - \langle \mu^*, \chi_R\Phi \rangle - \langle \mu^*, \bar{\chi}_R \Phi\rangle|
 \\&\leq |\mathbb{E}_{X_{\xi}(N\tau)}[(\chi_R\Phi) (X_{\xi}(t - N\tau))] -  \langle \mu^*, \chi_R\Phi \rangle | +
 | \mathbb{E}_{X_{\xi}(N\tau)}
  [(\bar{\chi}_{R}\Phi)(X_{\xi}(t-N\tau))] -\langle \mu^*, \bar{\chi}_R\Phi \rangle|\\&:= I_1+I_2 .
  \end{align*}
  Since $\chi_R\Phi$ vanishes outside of the ball $|x| \leq R+1$, we have
   \begin{align}\label{19}
  \sup_{x \in \mathbb{R}^d} |\chi_{R}(x)\Phi(x)|\leq \sup_{x\in\mathbb{R}^d, |x|\leq R+1} |\Phi(x)| \leq ||\Phi||_{bL, \gamma}e^{(R+1)^2}.
  \end{align}
  Let $$ \mathbb{S}:= \{(x_1, x_2) \in \mathbb{R}^d \times \mathbb{R}^d: |x_1| \leq R+1, |x_2| \geq R+1, 0 < |x_1 -x_2| \leq 1\}.$$
  It follows from $\chi_{R}(x_2) = 0, |x_2| \leq R+2$ and \eqref{19} that
\begin{align*}
\sup_{(x_1, x_2) \in \mathbb{S}}\frac{|(\chi_R\Phi)(x_1) - (\chi_R \Phi)(x_2)|}{|x_1-x_2|}
&=\sup_{(x_1, x_2) \in \mathbb{S}}\frac{|\chi_R(x_1)\Phi(x_1) - \chi_R(x_2) \Phi(x_1)|}{|x_1-x_2| }\\&
\leq 2||\chi_R||_{bL, \gamma}e^{(R+2)^2} ||\Phi||_{bL, \gamma} e^{(R+1)^2}
\leq 2||\chi_R||_{bL, \gamma} ||\Phi||_{bL, \gamma} e^{2(R+2)^2}.
\end{align*}

Let$$\mathbb{S}^R = \{(x_1, x_2) \in \mathbb{R}^d \times \mathbb{R}^d: |x_1| \leq R+1, |x_2| \leq R+1, 0 < |x_1 - x_2| \leq 1\}.$$
Then
\begin{align*}
\sup_{(x_1, x_2) \in \mathbb{S}^R}\frac{|(\chi_{R}\Phi)(x_1) - (\chi_R\Phi)(x_2)|}{|x_1 -x_2|}&=\sup_{(x_1, x_2) \in \mathbb{S}^R} \frac{|(\chi_{R}\Phi)(x_1) - \chi_{R}(x_1)\Phi(x_2) +\chi_{R}(x_1)\Phi(x_2)-(\chi_R\Phi)(x_2)|}{|x_1 -x_2|}\\&
\leq \sup_{(x_1, x_2) \in \mathbb{S}^R} \frac{|\chi_R(x_1)||\Phi(x_1)-\Phi(x_2)|}{|x_1 -x_2|}
+\frac{|\Phi(x_2)||\chi_{R}(x_1) - \chi_R(x_2)|}{|x_1 - x_2|}\\&\leq \sup_{(x_1, x_2)
\in \mathbb{S}^R} ||\Phi||_{bL, \gamma}e^{(R+1)^2}+2||\chi_R||_{bL, \gamma} ||\Phi||_{bL, \gamma}e^{2(R+1)^2} \\& \leq 6e^{2(R+1)^2}.
\end{align*}
So $\chi_R \Phi \in C_{bL}(\mathbb{R}^d)$ and $||\chi_R \Phi||_{bL} \leq 6e^{2(R+2)^2}$.

It is known that the dual Holder metric on $\mathcal{P}(\mathbb{R}^d) $ is bounded by the Wasserstein metric:
\begin{align*}
\sup_{\Phi \in C_{bL}(\mathbb{R}^d), ||\Phi||_{bL} \leq 1} |\langle \mu_1, \Phi\rangle - \langle \mu_2, \Phi\rangle| \leq 5d_L(\mu_1, \mu_2)
\end{align*}
for any $\mu_1, \mu_2 \in \mathcal{P}(\mathbb{R}^d)$.
And it also should be noted that
$$
\mathbb{E}_{X_{\xi}(N\tau)} \Phi(X_{\xi}(t - N\tau)) = P_{0, t-N\tau} \Phi(X_{\xi}(N\tau)). 
$$
Combing this with Theorem \ref{21} yields
\begin{align*}
I_1:&= |\mathbb{E}_{X_{\xi}(N\tau)}[(\chi_R\Phi) (X_{\xi}(t - N\tau))] -  \langle \mu^*, \chi_R\Phi \rangle |\\&=
|\langle P_{0, t - N\tau}^*\delta_{X_{\xi}(N\tau)}, \chi_{R}\Phi \rangle - \langle \mu^*, \chi_R\Phi \rangle| \\&
\leq 5\|\chi_R\Phi\|_{bL} d_L(P_{0, t-N\tau}^* \delta_{X_{\xi}(N\tau)}, \mu^*)\\&
= 5||\chi_R \Phi||_{bL} d_L(P_{0, t-N\tau}^* \delta_{X_{\xi}(N\tau)}, P_{0, t-N\tau}^* \mu^*)\\&
\leq 5 ||\chi_R\Phi||_{bL}2e^{|X_{\xi}(N\tau)|^2} d_L(\delta_{X_{\xi}(N\tau)}, \mu^*) e^{-\gamma(t-N\tau)}\\&
\leq C e^{2(R+2)^2} e^{|X_{\xi}(N\tau)|^2} e^{-\gamma(t -N\tau)},
\end{align*}
where \begin{align*}
d_{L}(\delta_{X_{\xi}(N\tau)}, \mu^*) &= d_L(\delta_{X_{\xi}(N\tau)}, \frac{\int_0^{\tau}\mu_s^*ds}{\tau})
= \sup_{Lip(f) \leq 1}|f(X_{\xi}(N\tau)) - \frac{\int_0^{\tau}\mathbb{E}f(X_{\xi}(s))ds}{\tau}|\\&= \sup_{Lip(f) \leq 1}
\frac{|\int_0^{\tau}\mathbb{E}[f(X_{\xi}(N\tau))- f(X_{\xi}(s))]ds|}{\tau} \leq \sup_{Lip(f) \leq 1}\frac{\int_0^{\tau}
\mathbb{E}|X_{\xi}(N\tau) - X_{\xi}(s)|ds}{\tau} < \infty.\end{align*}Here we used the fact from Theorem \ref{5}. Note that
\begin{align*}
I_2&:= |\mathbb{E}_{X_{\xi}(N\tau)}(\bar{\chi}\Phi)(X_{\xi}(t - N\tau)) -\langle \mu^*, \bar{\chi}_R\Phi \rangle|
\leq |\mathbb{E}_{X_{\xi}(N\tau)} (\bar{\chi}_{R}\Phi) (X_{\xi}(t - N\tau))| + |\langle \mu^*, \bar{\chi}_{R}\Phi \rangle|.
\end{align*}And due to Lemma \ref{22}, we have
\begin{align*}
|\mathbb{E}_{X_{\xi}(N\tau)} (\bar{\chi}_{R}\Phi) (X_{\xi}(t - N\tau))|
&\leq (\mathbb{E}_{X_{\xi}(N\tau)}|\bar{\chi}_R(X_{\xi}(t-N\tau))|^2)^{\frac{1}{2}}
(\mathbb{E}_{X_{\xi}(N\tau)}|\Phi(X_{\xi}(t-N\tau))|^2)^{\frac{1}{2}}\\&
\leq \left( \frac{\mathbb{E}_{X_{\xi}(N\tau)}e^{2|X_{\xi}(t-N\tau)|^2}}{e^{2R^2}}\right)^{\frac{1}{2}}
(\mathbb{E}_{X_{\xi}(N\tau)}e^{2|X_{\xi}(t-N\tau)|^2})^{\frac{1}{2}}\\&
\leq Ce^{2|\xi|^2}e^{-R^2},
\end{align*}
where the second inequality is from the Chernoff inequality: $$
\mathbb{P}(X \geq a) \leq \frac{\mathbb{E}e^{tX}}{e^{ta}}.$$ Then
\begin{align*}
\langle \mu^*, \bar{\chi}_{R} \Phi\rangle & \leq \left( \int_{\mathbb{R}^d}|\Phi(x)|^2 \mu^*(dx)\right)^{\frac{1}{2}}
\left( \int_{\mathbb{R}^d} \bar{\chi}_R(x)\mu^*(dx)\right)^{\frac{1}{2}} \\&
\leq (\mu^*(|x| \geq R)) \left( \int_{\mathbb{R}^d} e^{2|x|^2} \mu^*(dx)\right)^{\frac{1}{2}}||\Phi||_{bL, \gamma}\\&
\leq C e^{-R^2} e^{2|\xi|^2},
\end{align*}
\begin{align*}
|\mathbb{E}_{X_{\xi}(N\tau)}(\bar{\chi}_{R}\Phi)(X_{\xi}(t - N\tau)) -\langle \mu^*, \bar{\chi}_R\Phi \rangle|\leq
C e^{-R^2} e^{2|\xi|^2}.
\end{align*}
Hence\begin{align*}
|\mathbb{E}[\Phi(X_{\xi}(t))\mid \mathcal{F}_{N\tau} ] - \langle \mu^*, \Phi(\cdot)\rangle| &\leq
Ce^{-R^2}e^{2|\xi|^2} + Ce^{2(R+2)^2} e^{2|\xi|^2} e^{-\gamma(t-N\tau)}\\&
=  Ce^{2|\xi|^2}(e^{-R^2} +e^{2(R+2)^2-\gamma(t-N\tau)})
\leq Ce^{2|\xi|^2} e^{-\frac{\gamma (t-N\tau)}{{5}}}
\end{align*}
by choosing $R^2 = \frac{\gamma(t-N\tau)}{5}$.
 \end{proof}

 For $\Phi \in C_{bL}^{\gamma} (\mathbb{R}^d)$ above, define
 $$\Pi(N\tau) = \int_{N\tau}^{\infty}\mathbb{E}[\tilde{\Phi}(X_{\xi}(u)) \mid \mathcal{F}_{N\tau}] du =
  \int_{0}^{\infty}
 \mathbb{E}_{X_{\xi}(N\tau)}(\tilde{\Phi}(X_{\xi}(u)))du.$$
 For  $t \in \mathbb{R}^+, N = \lfloor \frac{t}{\tau} \rfloor$,  let
 \begin{align}\label{40}
 \int_0^t \tilde{\Phi}(X_{\xi}(u))du = \int_0^{N\tau}\tilde{\Phi}(X_{\xi}(u)) du
  +\int_{N\tau}^t \tilde{\Phi}(X_{\xi}(u)) du.
 \end{align}
From the approach of martingale approximation, we decompose the left of \eqref{40} into a martingale term $M_{N\tau}$ and a residual term $R_{N\tau, t}$:
 \begin{align*}
 M_{N\tau}  = \Pi(N\tau) - \Pi(0) +\int_0^{N\tau} \tilde{\Phi}(X_{\xi}(u)) du,
 \end{align*}
 \begin{align*}
 R_{N\tau, t} = -\Pi(N\tau) +\Pi(0) +\int_{N\tau}^t \tilde{\Phi}(X_{\xi}(u)) du,
 \end{align*}
 and we also define a martingale difference:
 \begin{align*}
 Z_{N} = M_{N\tau}- M_{(N-1)\tau}.
 \end{align*}
 \begin{lemma}\label{41}
 $\{M_{N\tau}\}_{N=1}^{\infty}$ is a martingale w.r.t. the filtration $\{\mathcal{F_{N\tau}}\}_{N = 1}^{\infty}$
with zero mean.
 \end{lemma}
\begin{proof}For $0 \leq k \leq N-1, k \in \mathbb{N}$,\begin{align*}
&\mathbb{E}[\Pi(N\tau) - \Pi(0) +\int_0^{N\tau} \tilde{\Phi}(X_{\xi}(u))du \mid \mathcal{F}_{k\tau}] \\&
=\mathbb{E}[\Pi(N\tau)\mid \mathcal{F}_{k\tau}] -\mathbb{E}[\Pi(0)\mid \mathcal{F}_{k\tau}]+\mathbb{E}[\int_{k\tau}^{N\tau} \tilde{\Phi}(X_{\xi}(u))du
\mid \mathcal{F}_{k\tau}] +\int_0^{k\tau}\tilde{\Phi}(X_{\xi}(u))du\\&
=\mathbb{E}[\int_{N\tau}^{\infty} \mathbb{E}[\tilde{\Phi}(X_{\xi}(u)) \mid \mathcal{F}_{N\tau}] \mid \mathcal{F}_{k\tau}]
-\Pi(0)+ \int_{0}^{k\tau}\tilde{\Phi}(X_{\xi}(u))du +\mathbb{E}[\int_{k\tau}^{N\tau}\tilde{\Phi}(X_{\xi}(u)) du\mid \mathcal{F}_{k\tau}]\\&=
\int_{N\tau}^{\infty} \mathbb{E}_{X_{\xi}(k\tau)}\tilde{\Phi}(\tilde{X}_{\xi}(u-k\tau))du -\Pi(0) + \int_0^{k\tau}
\tilde{\Phi}(X_{\xi}(u))du+ \int_{k\tau}^{N \tau} \mathbb{E}_{X_{\xi}(k\tau)}\tilde{\Phi}(X_{\xi}(u-k\tau))du\\&
=\int_{(N-k)\tau}^{\infty}\mathbb{E}_{X_{\xi}(k\tau)}\tilde{\Phi}(X_{\xi}(u))du - \Pi(0) +\int_{0}^{k\tau} \tilde{\Phi}(X_{\xi}(u))du
+\int_0^{(N-k)\tau} \mathbb{E}_{X_{\xi}(k\tau)}\tilde{\Phi}(X_{\xi}(u))du\\&=
\int_0^{\infty}\mathbb{E}_{X_{\xi}(k\tau)} \tilde{\Phi}(X_{\xi}(u))du -\Pi(0) +\int_0^{k\tau}\tilde{\Phi}(X_{\xi}(u))du
\\& =M_{k\tau}.
\end{align*}Besides,
\begin{align*}
\mathbb{E}M_{N\tau}&
=\mathbb{E}[\Pi(N\tau) - \Pi(0) +\int_0^{N\tau}\tilde{\Phi}(X_{\xi}(u))du]\\&=
\mathbb{E}[\int_{N\tau}^{\infty}\mathbb{E}[\tilde{\Phi}(X_{\xi}(u))\mid \mathcal{F}_{N\tau}]du]-\mathbb{E}\int_0^{\infty}
\tilde{\Phi}(X_{\xi}(u))du +\mathbb{E}\int_0^{N\tau}\tilde{\Phi}(X_{\xi}(u))du\\&
= \int_{N\tau}^{\infty}\mathbb{E}[\tilde{\Phi}(X_{\xi}(u))]du - \int_0^{\infty}\mathbb{E}
\tilde{\Phi}(X_{\xi}(u))du +\int_{0}^{N\tau}\mathbb{E}\tilde{\Phi}(X_{\xi}(u))du =0.
\end{align*}
\end{proof}

~~~~Then we obtain the following lemma, which plays a crucial role in the subsequent proof.
 \begin{lemma}\label{23}
 For $1\leq p \leq 2$, $N\in \mathbb{N}, \Phi \in C_{bL}^{\gamma}(\mathbb{R}^d),$
 \begin{align}\label{24}
 ~~~~~~~~~~~~~~~~\mathbb{E}|M(N\tau)|^{2^p} \leq C((N\tau)^{2-2^{-p}}+1) e^{2^{p+1}|\xi|^2} ,
 \end{align}
 \begin{align}\label{25}
 \mathbb{E}|Z_N|^{2^p} \leq Ce^{2^{p+1}|\xi|^2},
 \end{align}
 where $C$ does not depend on $N$.
 \end{lemma}
 \begin{proof}Note that
  \begin{align*}
 \mathbb{E}|M(N\tau)|^{2^p} \leq 3^{p-1}[\mathbb{E}|\Pi(N\tau)|^{2^p} + \mathbb{E}|\Pi(0)|^{2^p} +
 \mathbb{E}|\int_{N\tau}\tilde{\Phi}(X_{\xi}(u))du|^{2^p}].
 \end{align*}
 \begin{align*}
 \mathbb{E}|\Pi(N\tau)|^{2^p} &= \mathbb{E}|\int_0^{\infty} \mathbb{E}_{X_{\xi}(N\tau)}(\tilde{\Phi}(X_{\xi}(u)))du|^{2^p}\leq
 Ce^{2^{p+1}|\xi|^2} \int_0^{\infty} e^{-\frac{\gamma u 2^p}{5}} du\\&
 = Ce^{2^{p+1}|\xi|^2} \frac{5}{\gamma 2^{p}},
 \end{align*}
 where the first inequality is from Theorem \ref{20}.
 Similarly, we have
 \begin{align*}
 \mathbb{E}|\Pi(0)|^{2^p} \leq Ce^{2^{p+1}|\xi|^2} \frac{5}{\gamma 2^{p}},
 \end{align*}
 and \begin{align*}
 \mathbb{E}|\int_0^{N\tau} \tilde{\Phi}(X_{\xi}(u))du|^{2^p} &\leq (N\tau)^{1-2^{-p}} \int_0^{N\tau} \mathbb{E}|\tilde{\Phi}
 (X_{\xi}(u))|^{2^p} du\\&
 \leq (N\tau)^{1-2^{-p}}||\Phi||_{bL}^{2^p} \int_0^{N\tau}Ce^{2^{p+1}|\xi|^2} du\\&
 \leq (N\tau)^{2-2^{-p}}||\Phi||_{bL}^{2^p} Ce^{2^{p+1}|\xi|^2}.
 \end{align*}
 Then \begin{align*}
 \mathbb{E}|M(N\tau)|^{2^p}& \leq 3^{p-1}\left[ Ce^{2^{p+1}|\xi|^2}\frac{5}{\gamma 2^{p}} + Ce^{2^{p+1}|\xi|^2} \frac{5}{\gamma 2^{p}}
 + (N\tau)^{2-2^{-p}}||\Phi||_{bL, \gamma}^{2^p} e^{2^{p+1}|\xi|^2}  \right]\\& \leq C((N\tau)^{2-2^{-p}}+1) e^{2^{p+1}|\xi|^2}.
 \end{align*}
 Similar to the steps in the proof above, we can obtain
 \begin{align*}
 \mathbb{E}|Z_N|^{2^p} \leq Ce^{2^{p+1}|\xi|^2}
 \end{align*}
 for any $N \geq 1$, and $C$ does not depend on $N$.
 \end{proof}

 \begin{lemma}\label{26}
 For $\Phi \in C_{bL}^{\gamma}(\mathbb{R}^d), t \in \mathbb{R}^+$,
 \begin{align*}
 \lim_{t \rightarrow \infty} \frac{R_{N\tau, t}}{t} = 0, ~~~~\mathbb{P}-a.s..
 \end{align*}
\begin{proof}
Since $N = \lfloor \frac{t}{\tau} \rfloor$, it suffices to show \begin{align}\label{29}
\lim_{N \rightarrow \infty}\frac{1}{\sqrt{N}} \sup_{N\tau \leq t \leq (N+1)\tau} R_{N\tau, t} = 0, ~~~~\mathbb{P}-a.s..
\end{align}
Recall that \begin{align*}
|\Pi(N\tau)| &= |\int_0^{\infty}\mathbb{E}_{X_{\xi}(N\tau)}(\tilde{\Phi}(X_{\xi}(u)))du|\leq
\int_{0}^{\infty} C e^{2|X_{\xi}(N\tau)|^2}e^{- \frac{\gamma u}{5}}du\\& \leq C e^{2|X_{\xi}(N\tau)|^2} |\int_0^{\infty}
 e^{-\frac{\gamma u}{5}} du| = C e^{2|X(N\tau)|^2}\frac{5}{\gamma},\\
\sup_{N\tau \leq t \leq (N+1)\tau} |\int_{N\tau}^t \tilde{\Phi}(X(s))ds| &\leq C \tau \sup_{N\tau \leq t \leq (N+1)\tau}e^{2|X(t)|^2}.
\end{align*}
It then follows from Markov inequality that, for any $K > 0$,
\begin{align*}
\mathbb{P}\left(\sup_{N\tau \leq t \leq (N+1)\tau} e^{2|X(t)|^2} > K \right) \leq \frac{C e^{2^4 |\xi|^2}}{K^8}.
\end{align*}
Hence \begin{align*}
\sum_{N=1}^{\infty}\mathbb{P} &\left(\sup_{N\tau \leq t \leq (N+1)\tau}
(|\Pi(N\tau)| +\Pi(0) +|\int_{N\tau}^t \tilde{\Phi}(X(s))|) \geq N^{\frac{1}{4}}   \right)\\&
\leq \sum_{N=1}^{\infty}(C\tau \sup_{N\tau \leq t \leq (N+1)\tau} e^{2|X(t)|^2} \geq N^{\frac{1}{4}})
\leq Ce^{2^4|\xi|^2} \sum_{N=1}^{\infty} N^{-2} < \infty.
\end{align*}
By the Borel-Cantelli lemma, there is an almost surely finite random integer time $N_0(\omega)$ such that for $N
\geq N_0(\omega)$,
\begin{align*}
\sup_{N\tau \leq t \leq (N+1)\tau} R_{N\tau, t} \leq N^{\frac{1}{4}},
\end{align*}
which leads to  \begin{align*}
\lim_{N \rightarrow \infty}\frac{1}{\sqrt{N}} \sup_{N\tau \leq t \leq (N+1)\tau} R_{N\tau, t} = 0, ~~~~\mathbb{P}-a.s..
\end{align*}
\end{proof}
 \end{lemma}

  Now we arrive at one of the main results in this article, the strong law of large numbers.
\begin{theorem}
Suppose that (H3)-(H9) hold. For $\Phi \in C_{bL}^{\gamma}(\mathbb{R}^d)$, $ \epsilon > 0$,\begin{align}\label{33}
\lim_{t \rightarrow \infty}\frac{\int_0^t \tilde{\Phi}(X_{\xi}(u))du}{t^{\frac{1}{2}+ \epsilon}} = 0,~~~~\mathbb{P}-a.s..
\end{align}
\end{theorem}
\begin{proof}
In view of Lemma \ref{23}, Lemma \ref{26} above, and Theorem \ref{27} in appendix, we get the desired \eqref{33}. To be specific,
 let $p=1$ in Lemma \ref{23}, then we meet the condition  in Theorem \ref{27} by choosing $C_{N} = N^{\frac{1}{2} + \epsilon}$,
 $R_{N \tau, t}$ is the residual term, which is negligible, then we get \eqref{33} from the construction of $M_{N\tau}$ and $R_{N\tau, t}$ .
\end{proof}
  \section{Central limit theorem}

Before proving the key theorem, we do some preparation.
 Let $a_n(\xi) = \sum_{k=0}^n P_{0, k\tau}\tilde{\Phi}(\xi)$, we have the following lemma.\begin{lemma}
  The sequence of functions $a_n$ converges uniformly on bounded sets. The point wise limit
  $$
  a:= \lim_{n \rightarrow \infty} a_n$$ is a Lipschitz function and for any $n \geq m$,
  \begin{align*}
  \mathbb{E}[a(X_{\xi}(n\tau))\mid \mathcal{F}_{m\tau}] = \lim_{N \rightarrow \infty} \mathbb{E}[a_N(X_{\xi}(n\tau))
   \mid \mathcal{F}_{m\tau}].
  \end{align*}
  \end{lemma}
\begin{proof}
Thanks to Theorem \ref{21}, for any $\epsilon > 0$, we have
\begin{align*}
|a_m(\xi) - a_n(\xi)| &= |\sum_{k = m}^n P_{0, k\tau} \Phi(\xi) - \langle \mu^*, \Phi(\cdot) \rangle| \leq
\sum_{k=m}^n Lip(\Phi)d_L(\delta_{\xi}, \mu^*) e^{-\gamma k\tau}\\& =C Lip(\Phi) d_L(\delta_{\xi}, \mu^*)\frac{e^{-\gamma m\tau}
- e^{-\gamma n\tau}}{1-e^{-\gamma \tau}}     < \epsilon
\end{align*}
for $n \geq m > N$ and $N$ is large enough. Hence $a_n$ converges uniformly on bounded sets. Note that
\begin{align*}
|a_N(\xi_1) - a_{N}(\xi_2)| \leq \sum_{k=0}^{N} |\langle P_{0, k\tau}^* \delta_{\xi_1} - P_{0, k\tau}^*
 \delta_{\xi_2}, \Phi \rangle| \leq C Lip(\Phi)|\xi_1 - \xi_2| \sum_{k=0}^{N} e^{-\gamma k\tau},
\end{align*}
hence $$
|a(\xi_1) - a(\xi_2)| \leq C Lip(\Phi)|\xi_1 -\xi_2|,$$
where $C$ does not depend on $\xi_1, \xi_2 \in \mathbb{R}^d$. So $a$ is also a Lipschitz function. Since $a_N$ and $a$
are Lipschitz and $$ d_L(P_{0, k\tau}^* \delta_{\xi}, \delta_0) < \infty,
$$ we know that they are $P_{0, k\tau}^*\delta_{\xi}$ integrable. By the dominated convergence theorem for any $\xi
\in \mathbb{R}^d$,
\begin{align*}
\lim_{N \rightarrow \infty}P_{0, (n-m)\tau}a_{N}(\xi) &= \lim_{N \rightarrow \infty} \mathbb{E}_{\xi}[a_N(X_{\xi}((n-m)\tau))]\\&
=\lim_{N \rightarrow \infty}\langle P_{0, (n-m)\tau}^*\delta_{\xi}, a_N\rangle = \langle P_{0, (n-m)\tau}^*\delta_{\xi}, a\rangle
= P_{0, (n-m)\tau}a(\xi).
\end{align*}
Hence
\begin{align*}
\lim_{N  \rightarrow \infty}\mathbb{E}_{\xi}[a_N(X_{\xi}(n\tau)) \mid \mathcal{F}_{m\tau}] = \lim_{N \rightarrow \infty}
 P_{0, (n-m)\tau} a_{N}(X_{\xi}(m\tau)) = \mathbb{E}[a(X_{\xi}(n\tau)) \mid \mathcal{F}_{m\tau}].
\end{align*}
\end{proof}

  In order to obtain the key theorem of this section, we need a crucial theorem
  from \cite{ref20}, that is Theorem \ref{28} in appendix, which is based on the martingale approximation approach.
   Now we are in the position to verify the conditions (M1)-(M3) in Theorem \ref{28}, which lead to the central limit theorem directly. It should also be noted that $$
   \nu_N:= \frac{1}{N} \sum_{j=1}^N P_{0, (j-1)\tau}^*\delta_{\xi} \overset{w}{\rightarrow} \mu^*.$$
   \begin{theorem}\label{34}
Suppose that (H3)-(H9) hold. For $ \Phi \in C_{bL}^{\gamma}(\mathbb{R}^d)$,

$$
\lim_{t \rightarrow \infty} \frac{1}{\sqrt{t}} \int_0^t \tilde{\Phi}(X_{\xi}(t)) dt \overset{\mathcal{D}}{=} N(0, \sigma^2),
$$
where $N(0, \sigma^2)$ denotes the standard normal random variable with variance $\sigma^2$ . In fact,
\begin{align*}
\sigma^2 = \langle \mu^*, \mathbb{E}_{\xi}M_{\tau}^2\rangle.
\end{align*}
\end{theorem}
   \begin{proof}Set $F(\xi) = \mathbb{E}_{\xi}[Z_1^2 \mid |Z_1| \geq \epsilon \sqrt{N}] = \mathbb{E}_{\xi}[(M_{\tau} - M_0)^2
 \mid |M_{\tau}| \geq \epsilon \sqrt{N}]$. By the Markov property, we have
 \begin{align*}
 P_{0, (k-1)\tau} F(\xi) &= \mathbb{E}_{\xi}[F(X_{\xi}((k-1)\tau))]\\& =
 \mathbb{E}_{\xi}[\mathbb{E}_{X_{\xi}((k-1)\tau)}[Z_1^2 \mid |Z_1| \geq \epsilon \sqrt{N} ]]\\&\
 = \mathbb{E}_{\xi}[Z_{k}^2 \mid |Z_k| \geq \epsilon \sqrt{N}].
 \end{align*}
 Hence $$\frac{1}{N} \sum_{j=0}^{N-1}\mathbb{E}_{\xi}[Z_{j+1}^2 \mid |Z_{j+1}| \geq \epsilon \sqrt{N}] =
  \langle \frac{1}{N} \sum_{j=1}^N P_{0, (j-1)\tau} \delta_{\xi}, F \rangle.$$
   To verify (M1), it remains to show that $
   F_N(\xi) = F(\xi)$ converges to $0$ uniformly om compact sets and there is some $\eta > 0$ that satisfy \eqref{30}.

   Note that for any $\delta > 0$,
   \begin{align*}
   F_N(\xi)&:= \mathbb{E}_{\xi}[Z_1^2 \mid |Z_1| \geq \epsilon \sqrt{N}] \leq
   (\mathbb{E}_{\xi}|Z_1|^{2+\delta})^{\frac{2}{2+\delta}} \mathbb{P}(|Z_1|
    \geq \epsilon \sqrt{N})^{\frac{\delta}{2+\delta}}\\&\leq (\mathbb{E}_{\xi}|Z_1|^{2+\delta})^{\frac{2}{2+\delta}}
    \left(\frac{\mathbb{E}_{\xi}|Z_1|^{2+\delta}}{\epsilon^{2+\delta} N^{\frac{2+\delta}{2}}}\right)^{\frac{\delta}{2+\delta}}
     = \frac{\mathbb{E}_{\xi}|Z_1|^{2+\delta}}{\epsilon^{\delta} N^{\frac{\delta}{2}}}.\\
   \mathbb{E}_{\xi}|Z_1|^{2+\delta} &= \mathbb{E}_{\xi}|\Pi(\tau) - \Pi(0) +\int_0^{\tau} \Phi(X_{\xi}(u))du|^{2+\delta}
   \leq (\mathbb{E}_{\xi}|Z_1|^4)^{\frac{2+\delta}{4}}\\&\leq (Ce^{8|\xi|^2})^{\frac{2+\delta}{4}}= Ce^{2(2+\delta)|\xi|^2} < \infty.
   \end{align*}
   Then $$
   \sup_{\xi \in \mathbb{R}^d}F_N(\xi) \leq \frac{Ce^{2(2+\delta)|\xi|^2}}{\epsilon^{\delta}N^{\frac{\delta}{2}}}
    \overset{N \rightarrow \infty}{\rightarrow} 0.$$
Hence $F_N(\xi)$ converges to 0 uniformly on any compact sets. And
\begin{align*}
\lim_{N \rightarrow \infty}\langle \nu_N, |F_N|^{1+\frac{\delta}{2}} \rangle &\leq \lim_{N \rightarrow \infty}\frac{1}{N}
\sum_{j=1}^N P_{0, (j-1)\tau} Ce^{(2+\delta)^2\xi^2}\\& \leq \lim_{N \rightarrow \infty} \frac{1}{N} \sum_{j=1}^{N}
\mathbb{E}e^{(2+\delta)^2 X_{\xi}((j-1)N\tau)} < \infty.
\end{align*}
From \cite{ref20}, we know that if $\nu_N$ converges to $\mu^*$ weakly, $F_N \rightarrow 0$ uniformly on compact sets and there is an $\eta > 0$
   such that \begin{align}\label{30}
   \lim_{N \rightarrow \infty}\langle \nu_N, |F_N|^{1+\eta} \rangle < \infty,\end{align} then \begin{align}\label{31}
   \lim_{N \rightarrow \infty} \langle \nu_N, F_N \rangle =0,\end{align}which leads to (M1) directly.

   For (M2), by periodicity and Markov property, for any $\sigma \geq 0$,
   \begin{align*}
   \frac{1}{l}\sum_{m=1}^l \mathbb{E}_{\xi}\left|\frac{1}{K}\mathbb{E}_{\xi} [[M]_{mK\tau} - [M]_{(m-1)K\tau} \mid
   \mathcal{F}_{(m-1)K\tau}]-\sigma^2\right|= \frac{1}{l}\sum_{m=1}^l P_{0, (m-1)K\tau}|H_K(\xi)|,
   \end{align*}
   where $H_K(\xi) = \mathbb{E}_{\xi}\left[\frac{1}{K}[M]_{K\tau}  - \sigma^2  \right] = \mathbb{E}_{\xi}
   [\frac{1}{K}M_{K\tau}^2 - \sigma^2]$. From Lemma \ref{23}, we know that
   \begin{align*}
   \limsup_{l \rightarrow \infty} \frac{1}{l}\sum_{m=1}^l \langle P_{0, (m-1)K\tau}^* \delta_{\xi}, |H_K| \rangle < \infty.
   \end{align*}
   Hence \begin{align*}
   \lim_{l \rightarrow \infty}\frac{1}{l}\sum_{m=1}^l \langle P_{(m-1)K\tau}^* \delta_{\xi}, |H_K| \rangle =
   \langle \mu^* ,|H_K|\rangle = \mathbb{E}\left|\frac{1}{K} \sum_{j=0}^{K-1}P_{0, j\tau} J(\xi) \right|,
   \end{align*}
   where $J_{\xi} = \mathbb{E}_{\xi} M_{\tau}^2 - \sigma^2$.

   Since $\mu^*$ is ergodic under the Markov semigroup $\{P_{0, k\tau}\}_{k\geq 0}$, by the Birkhoff theorem,
   \begin{align*}
   \frac{1}{K}\sum_{j=0}^{K-1} P_{0, j\tau}(\mathbb{E}_{\xi}M_{\tau}^2) = \frac{1}{K}\sum_{j=0}^{K-1} \mathbb{E}_{X_{\xi}(j\tau)}M_{\tau}^2
   = \int_{\mathbb{R}^d} \mathbb{E}_z M_{\tau}^2 \mu^*(dz), ~~~~K \overset{a.s.}{\rightarrow} \infty.\end{align*}
   Hence if we choose $\sigma^2 = \langle \mu^*, \mathbb{E}_{\xi}M_{\tau}^2\rangle$, then$$ \langle \mu^*, |H_K| \rangle \rightarrow 0,~~~~ K\rightarrow \infty.$$

Finally, we consider the condition (M3). By the periodicity and Markov property, we rewrite the expression in
condition (M3) as follows:
\begin{align*}
\sum_{m=1}^l \sum_{j=(m-1)K}^{mK-1} \mathbb{E}_{\xi}[1+Z_{j+1}^2 \mid |M_{j\tau} -M_{(m-1)K\tau}| \geq \epsilon \sqrt{lK}]
= \frac{1}{K}\sum_{j=0}^{K-1} \langle Q_{l, K}^*\delta_{\xi} , G_{l, j}\rangle,
\end{align*} where
$$\langle Q_{l, K}^*\delta_{\xi} , G_{l, j}\rangle = \frac{1}{l} \sum_{m=1}^l \langle P_{0, (m-1)K\tau}^* \delta_{\xi}, G_{l, j} \rangle,$$and
$$G_{l, j}(\xi) = \mathbb{E}_{\xi} [1+Z_{j+1}^2 \mid |M_{j\tau}| \geq \epsilon \sqrt{lK}].$$ To verify (M3), we now to
prove \begin{align}\label{32}
\limsup_{l \rightarrow \infty} \langle Q_{l, K}^* \delta_{\xi}, G_{l, j} \rangle = 0, \end{align} for $j= 0, 1, \cdots, K-1.$
In fact,
\begin{align*}
G_{l,j}(\xi)& \leq (\mathbb{E}(1+Z_{j+1}^2)^2)^{\frac{1}{2}} \mathbb{P}(|M_{j\tau}| \geq \epsilon \sqrt{lK})\\&
\leq(\mathbb{E}(1+Z_{j+1}^2)^2)^{\frac{1}{2}} \left(\frac{\mathbb{E}|M_{j\tau}|^2}{(\epsilon \sqrt{lK})^2}\right)^{\frac{1}{2}}\\&
=(\mathbb{E}(1+Z_{j+1}^2)^2)^{\frac{1}{2}} \frac{(\mathbb{E}|M_{j\tau}|^2)^{\frac{1}{2}}}{\epsilon \sqrt{lK}}\\&
\leq C(\epsilon ) \frac{1+\mathbb{E}_{\xi}|M_{(j+1)\tau}|^4 + |M_{j\tau}|^4}{\epsilon \sqrt{lK}}.
\end{align*}
In view of Lemma \ref{23}, for any $R > 0$, $0 \leq j \leq K$,
\begin{align*}
\sup_{\xi \in B_{R}(0)} G_{l, j}(\xi) &\leq C(\epsilon)\frac{1+C(((j+1)\tau)^{2-2^{-2}}+1)e^{8|\xi|^2} +
 C(1+(j\tau)^{2-2^{-2}})e^{8|\xi|^2}}{\epsilon \sqrt{lK}}\\& \leq \frac{ C( \epsilon)}{\sqrt{lK}},
\end{align*}
where $B_{R}(0)$ denotes the ball with center $0$ and radius $R$ in $\mathbb{R}^d$.
Then we have
\begin{align*}
\langle \frac{1}{l} \sum_{m=1}^l P_{0, (m-1)K\tau}^* \delta_{\xi}, G_{l, j}^2 \rangle \leq C(K, \xi),
\end{align*}and \begin{align*}
\limsup_{l \rightarrow \infty} \langle \frac{1}{l} \sum_{m=1}^l P_{0, (m-1)K\tau}^* \delta_{\xi}, G_{l, j}^2 \rangle  <
\infty.
\end{align*}
So
\begin{align*}
\limsup_{l \rightarrow \infty} \langle \frac{1}{l} \sum_{m=1}^l P_{0, (m-1)K\tau}^* \delta_{\xi}, G_{l, j}\rangle  =0,
\end{align*} for $j = 0, 1, \cdots, K-1$. Then \eqref{32} is achieved.
\end{proof}

 \section{Appendix}
 \begin{definition}\label{3}
An $\mathbb{R}^d$-valued stochastic process $L= (L(t), t \geq 0)$ is called L{\'e}vy process, if

(1) L(0) = 0, a.s.;

(2) L has independent and stationary increments;

(3) L is stochastically continuous, i.e., for any $\epsilon > 0$ and $s \geq 0$,
\begin{align*}
\lim_{t \rightarrow s} \mathbb{P}(|L(t) - L(s)| > \epsilon) = 0.
\end{align*}
\end{definition}
\begin{remark}
Under (1) and (2), (3) can be expressed by:
\begin{align*}
\lim_{t \downarrow 0} \mathbb{P} (|L(t)| > \epsilon) = 0
\end{align*}
for all $\epsilon > 0$.
\end{remark}

\begin{proposition}\label{2} (L{\'e}vy-It{\^o} composition)
If $L$ is an $\mathbb{R}^d$-valued L{\'e}vy process, then there exist $b \in \mathbb{R}^d$, an $\mathbb{R}^d$-valued Wiener process $W$ with covariance operator $Q$, the so-called $Q$-Wiener process, and an independent Poisson random measure $N$ on $(\mathbb{R}^+ \cup \{0\} \times (\mathbb{R}^d - \{0\})$ such that: for each $t \geq 0$,
\begin{align*}
L(t) = bt +W(t) +\int_{|x| < 1} x N_1(t, dx) + \int_{|x| \geq 1} x N(t, dx).
\end{align*}
Here the Poisson random measure $N$ has the intensity measure $\nu$ which satisfies
\begin{align}\label{4}
\int_{\mathbb{R}^d} (|y|^2 \wedge 1) \nu(dy) < \infty,
\end{align}
and $N_1$ is the compensated Poisson random measure of $N$.
\end{proposition}
Suppose that $X(t), t_0 \leq t \leq \tau$ is the unique solution of \begin{align*}
dX(t) = f(t, X(t))dt + g(t, X(t)) dW(t)
\end{align*}
with initial value $\xi \in \mathcal{L}^p(\mathbb{P}, \mathbb{R}^d)$. Then there are some existed results about the pth moment
of the solution. For more details and results, we can refer to \cite{ref31}, and we present some of which in the lemmas below.
\begin{lemma}\label{7}
Suppose that there exists a constant $\alpha > 0$,
such that for all $p \geq 2, \xi \in \mathcal{L}^p(\mathbb{P}, \mathbb{R}^d), t \in [0, \tau]$, $$
\langle x, f(t, x) \rangle +\frac{p-1}{2}|g(t, x)|^2 \leq \alpha (1+|x|^2).
$$
Then
$$
\mathbb{E}|X(t)|^p \leq 2^{\frac{p-2}{2}}(1+\mathbb{E}|\xi|^p)e^{p\alpha t}
$$
for all $t \in [0, \tau]$.
\end{lemma}

\begin{lemma}\label{8}
Let $p \geq 2$, for any $\tau > 0$, $$
\mathbb{E}\int_0^\tau |g(s)|^2 < \infty.$$
Then \begin{align*}
\mathbb{E}\left|\int_0^{\tau}g(s) dW(s)\right|^p \leq \left(\frac{p(p-1)}{2}\right)^{\frac{p}{2}} \tau^{\frac{p-2}{2}}
\mathbb{E}\int_0^{\tau}|g(s)|^p ds.
\end{align*}
In particular, if $p=2$, this is an equality.
\end{lemma}
More generally, we have the following lemma.
\begin{lemma}\label{9}
Let $p \geq 2$, for any $\tau > 0$, $$
\mathbb{E}\int_0^\tau |g(s)|^2 < \infty.$$
Then \begin{align*}
\mathbb{E}\left(\sup_{0 \leq t \leq \tau}|\int_0^t g(s) dW(s)|^p \right)\leq \left(\frac{p^3}{2(p-1)}\right)^{\frac{p}{2}} \tau^{\frac{p-2}{2}}
\mathbb{E}\int_0^{\tau} |g(s)|^p ds.
\end{align*}
\end{lemma}
Similar to the results above, we have the following lemma.
\begin{lemma}\label{10}
Let $p \geq 2$, for any $\tau > 0$, $$
~~~~~~~~\mathbb{E}\int_0^\tau\int_{|x|< 1} |f(r, x)|^p \nu(dx)dr < \infty.$$
Then \begin{align*}
~~~~~~~~\mathbb{E}\left|\int_0^\tau \int_{|x| < 1} f(r, x) N_1(dr, dx)  \right|^2 &=\mathbb{E}\int_0^{\tau}\int_{|x|< 1}
|f(r, x)|^2 \nu(dx)dr
;\\
\mathbb{E}\left|\int_0^\tau \int_{|x| < 1} f(r, x) N_1(dr, dx)  \right|^p &\leq \left(\frac{p(p-1)}{2}\right)^{\frac{p}{2}} \tau^{\frac{p-2}{2}}
\mathbb{E}\int_0^{\tau}\int_{|x|< 1}|f(r, x)|^p \nu(dx)dr;
\end{align*}
and
\begin{align*}
\mathbb{E}\left(\sup_{0 \leq t \leq \tau}\left| \int_0^t\int_{|x|< 1}f(s, x)\tilde{N}(ds, dx)\right|^p\right) \leq  \left(\frac{p^3}{2(p-1)}\right)^{\frac{p}{2}} \tau^{\frac{p-2}{2}} \mathbb{E}\int_0^{\tau}\int_{|x|< 1}
|f(r, x)|^p \nu(dx)dr.
\end{align*}
\end{lemma}

\begin{lemma}\label{11}
Let $\{Y_n\}_{n=1}^{\infty}$ be a sequence of stochastic process in $\mathbb{R}^d$ such that for any $\{t_i\}_{i=1}^k \subset \mathbb{R}^+$, the joint distribution of $\{Y_n(t_i)\}_{i=1}^k$ is weakly convergent as $n \rightarrow \infty$,
and the sequence of $\{Y_n\}_{n=1}^{\infty}$ is uniformly stochastic continuous, that is
$$
\sup_{n, |s_1 - s_2| < \delta t} \mathbb{P}\{|Y_n(s_1) - Y_n(s_2)| > \epsilon\} \rightarrow 0, ~~~~\delta \rightarrow 0.$$
Then a sequence of stochastic processes $\{\tilde{Y}_n\}_{n=1}^{\infty}$ can be constructed in another probability space $(\tilde{\Omega}, \tilde{\mathcal{F}}, \tilde{P})$ such that

(1) $\tilde{Y}_n \overset{\mathcal{D}}{\rightarrow} \tilde{Y}$;

(2) The finite-dimensional distributions of the processes $Y_n$ and $\tilde{Y}_n$ coincide for $n > 0$
 \end{lemma}

 \begin{lemma}\label{12}
 Suppose that the sequence of stochastic processes $\{Y_n\}_{n=1}^{\infty}$ is uniformly stochastic continuous, and uniformly bounded in probability, that is$$
 \sup_{t, n}\mathbb{P}\{Y_n(t) > M\} \rightarrow 0, ~~~~ M\rightarrow \infty.$$
 Then $\{Y_n\}_{n=1}^{\infty}$ contains a subsequence $\{Y_{n_k}\}_{k=1}^{\infty}$ with weakly convergent finite-dimensional distributions.  \end{lemma}
\begin{lemma}\label{39}
Let $p > 1$, and $\mathcal{A}$ be a family of integrable random variables in $\mathcal{L}^p(\Omega, \mathbb{R}^d)$. If there
exists a positive constant $C $ such that
$$
\sup_{\xi \in \mathcal{A}} \mathbb{E}|\xi|^p :=C < \infty,$$
then
\begin{align}\label{14}
\sup_{\xi \in \mathcal{A}} \int_{|\xi| \geq M}|\xi| d \mathbb{P}\rightarrow 0,~~~~M \rightarrow \infty.
\end{align}
\end{lemma}
\begin{remark}\label{15}
Let $\mathcal{A} \in \mathcal{L}^1(\Omega, \mathbb{R}^d)$, then \eqref{14} is equivalent to:

(1) $a_0 = \sup_{\xi \in \mathcal{A}}\{\mathbb{E}|\xi|\} < \infty$;

(2) For any $\epsilon > 0$, there exists a $\delta > 0$ such that $$
\sup_{\xi \in \mathcal{A}}\int_{A}|\xi| d \mathbb{P} \leq \epsilon,$$
where $A \in \mathcal{F}$ satisfying $\mathbb{P}(A) \leq \delta$.
\end{remark}
Let we give a corollary of Lebesgue's dominated convergence theorem.
\begin{lemma}\label{16}
Let $\eta, \xi, \xi_1, \cdots$ be random variables such that $|\xi_n| \leq \eta$, $\xi_n \overset{a.s.}{\rightarrow} \xi$ and
$\mathbb{E}\eta^p < \infty$ for some $p > 0$. Then $\mathbb{E}|\xi|^p < \infty$ and $\mathbb{E}|\xi - \xi_n|^p \rightarrow 0$
as $n \rightarrow \infty$.
\end{lemma}
\begin{lemma}\label{18}
For any $t \geq 0$, if there is an $\alpha > 0$, such that for any $x_1, x_2 \in \mathbb{R}^d$,
$$
d_{L}(P_{0, t}^* \delta_{x_1}, P_{0, t}^*\delta_{x_2}) \leq \alpha|x_1 - x_2|,$$
then for any $\mu_1, \mu_2 \in \mathcal{P}(\mathbb{R}^d)$,
\begin{align*}
d_L(P_{0, t}^*\mu_1, P_{0, t}^*\mu_2) \leq \alpha d_L(\mu_1, \mu_2).
\end{align*}
\end{lemma}

\begin{theorem}\label{27}(\cite{ref17}\cite {ref22})
  Let $\{M_{N\tau}\}_{N \geq 1}$ be a zero mean square integrable martingale and let $\{C_N\}_{N \geq 1}$ be an increasing sequence
  going to $\infty$ such that $$
  \sum_{N=1}^{\infty} C_{N}^{-2} \mathbb{E} Z_{N}^2 < \infty,$$
  where $Z_N = M_{N\tau} - M_{(N-1)\tau}$ and $M_0 = 0$. Then
  \begin{align*}
  \lim_{N \rightarrow \infty} C_{N}^{-1} M_{N\tau} = 0, ~~~~\mathbb{P}-a.s..
  \end{align*}
  \end{theorem}

 \begin{theorem}\label{28}
  Assume that the martingale $M_{N\tau}$, its quadratic variation $[M]_{N\tau}$ and the associated martingale difference
  $Z_N = M_{N\tau} - M_{(N-1)\tau}$ (with $M_0 = 0$) satisfy the following: for every $\epsilon > 0$,

  [M1] $\lim_{N \rightarrow \infty} \frac{1}{N} \sum_{j=0}^{N-1} \mathbb{E}[Z_{j+1}^2 \mid |Z_{j+1}| \geq
  \epsilon \sqrt{N}] = 0$.

 [M2] There exists $\sigma > 0$ such that
  \begin{align*}
  \lim_{K \rightarrow \infty} \limsup_{l \rightarrow \infty} \frac{1}{l} \sum_{m=1}^l \mathbb{E}|\frac{1}{K} \mathbb{E}
  [[M]_{mK\tau} - [M]_{(m-1)K\tau}\mid \mathcal{F}_{(m-1)K\tau}] - \sigma^2 | =0
  \end{align*}
  with uniformly square integrable condition $\sup_{N \geq 1} \mathbb{E}Z_{N}^2 < \infty$.

  [M3] $\lim_{K \rightarrow \infty} \limsup_{l \rightarrow \infty} \frac{1}{lK} \sum_{m=1}^l \sum_{j=(m-1)K}^{mK-1}
  \mathbb{E}[1+Z_{j+1}^2 \mid |M_{j\tau} -M_{(m-1)K\tau}| \geq \epsilon \sqrt{lK}] =0.$

  Then one has $$\lim_{N \rightarrow \infty}\frac{\mathbb{E}[M]_{N\tau}}{N} = \sigma^2$$
  and $$ \lim_{N \rightarrow \infty} \mathbb{E}e^{i\theta \frac{M_{N\tau}}{\sqrt{N}}} = e^{-\frac{\sigma^2 \theta^2}{2}}$$
for any $\theta \in \mathbb{R}$.
  \end{theorem}
\section*{Acknowledgment}
This work is supported by  National Natural Science Foundation of China (12071175, 12371191).

\end{document}